\documentclass[reqno]{amsart}
\usepackage[margin=1.3in]{geometry}
\usepackage{amssymb, amsfonts}
\usepackage{enumerate}
\usepackage[usenames, dvipsnames]{color}
\usepackage{verbatim}
\usepackage[hypertex]{hyperref}
\usepackage{mathrsfs}

\numberwithin{equation}{section}

\newtheorem{theorem}{Theorem}[section]
\newtheorem{corollary}[theorem]{Corollary}
\newtheorem{lemma}[theorem]{Lemma}
\newtheorem{prop}[theorem]{Proposition}

\theoremstyle{definition}

\theoremstyle{definition}

\theoremstyle{definition}

\makeatletter
\def\dashint{\operatorname%
{\,\,\text{\bf--}\kern-.98em\DOTSI\intop\ilimits@\!\!}}
\makeatother

\newcommand{\abs}[1]{\lvert#1\rvert}
\newcommand{\norm}[1]{\lVert#1\rVert}

\def\R{\mathbb{R}}      
\def\W{\Omega}          
\def\w{\omega}          
\def\esssup{\text{ess sup}}  

\def\bR{\mathbb{R}}

\begin{document}

\title[the Naiver-Stokes equations]{Interior and boundary regularity for the Navier-Stokes equations in the critical Lebesgue spaces}

\
\author[H. Dong]{Hongjie Dong}
\address[H. Dong]{Division of Applied Mathematics, Brown University, 182 George Street, Providence, RI 02912, USA}

\email{Hongjie\_Dong@brown.edu}

\thanks{H. Dong was partially supported by the NSF under agreement DMS-1600593.}

\author[K. Wang]{Kunrui Wang}
\address[K. Wang]{Division of Applied Mathematics, Brown University, 182 George Street, Providence, RI 02912, USA}

\email{Kunrui\_Wang@brown.edu}

\thanks{K. Wang was partially supported by the NSF under agreement DMS-1600593.}

\subjclass[2010]{Primary }

\date{\today}




\begin{abstract}
We study regularity criteria for the $d$-dimensional incompressible Navier-Stokes equations. We prove if $u\in L_{\infty}^tL_d^x((0,T)\times\R^d_+)$ is a Leray-Hopf weak solution vanishing on the boundary and the pressure $p$ satisfies a local condition $\|p\|_{L_{2-1/d}(Q(z_0,1)\cap (0,T)\times \R^d_+)}\leq K$ for some constant $K>0$ uniformly in $z_0$, then $u$ is regular up to the boundary in $(0,T)\times \R^d_+$. Furthermore, when $T=\infty$, $u$ tends to zero as $t\rightarrow \infty$. We also study the local problem in half unit cylinder $Q^+$ and prove that if $u\in L^t_{\infty}L^x_d(Q^+)$ and $ p\in L_{2-1/d}(Q^+)$, then $u$ is H\"{o}lder continuous in the closure of the set $Q^+(1/4)$. This generalizes a result by Escauriaza, Seregin, and \v{S}ver\'{a}k to higher dimensions and domains with boundary.
\end{abstract}
\maketitle

\section{Introduction}
In this paper we discuss the incompressible Navier-Stokes equations in $d$ spatial dimension with unit viscosity and zero external force:
\begin{equation}\label{NS1}
\partial_tu+u\cdot\nabla u- \Delta u +\nabla p=0, \quad \text{div } u=0
\end{equation}
for $x\in\W$ and $t> 0$ with the initial condition
\begin{equation}\label{NS2}
u(0,x) = a(x)  \text{  in } \W.
\end{equation}
Here $u$ is the velocity and $p$ is the pressure. We consider three kinds of domains: the whole space $\W: =\mathbb{R}^d$, the half space  $\W: =\mathbb{R}^d_+$, and the half cylinder $\W: = Q^+$.

For both $\W = \mathbb{R}^d_+$ and $Q^+$, we assume that $u$ satisfies the zero Dirichlet boundary condition:
\begin{equation}\label{NS3}
u = 0  \text{  on }   \{x_d=0\}\cap\partial\W.
\end{equation}

For $d=3$, the global existence of strong solutions to the Navier-Stokes equations has long been recognized as an important fundamental problem in fluid dynamics and is still widely open.  The local solvability, assuming a sufficiently regular initial data $a$, is well known  (see \cite{Kato1, Giga1, Taylor1, Koch1}). Indeed, the local solution is unique and smooth in both spatial and time variables.

This paper starts from another important type of solutions called {\sl Leray-Hopf} weak solutions. See Section \ref{LH_def} for the notation and definition. In the pioneering works of Leray and Hopf, it is shown that for any divergence-free vector field $a\in L_2$, there exists at least one Leray-Hopf weak solution of the Cauchy problem (\ref{NS1})-(\ref{NS2}) on $(0,\infty)\times\R^d$. Although the problems of uniqueness and regularity of Leray-Hopf weak solutions are still open, since the seminal work of Leray and Hopf, there are extensive literatures on conditional results under various criteria. The most well-known condition is the so-called Ladyzhenskaya-Prodi-Serrin condition, which requires for some $T>0$,
\begin{equation}
\label{serrin1}
u\in L^t_rL^x_q(\R^{d+1}_T),
\end{equation}
where the pair $(r,q)$ satisfies
\begin{equation}
\label{serrin2}
\frac{2}{r}+\frac{d}{q}\leq 1,\quad q\in (d,\infty].
\end{equation}
Under the conditions (\ref{serrin1}) and (\ref{serrin2}), the uniqueness of Leray-Hopf weak solutions was proved by Prodi \cite{Refer22} and Serrin \cite{Refer28}, and the smoothness was obtained by Ladyzhenskaya \cite{Refer15}. For further results, we refer the reader to \cite{Refer8,Refer30,Refer31,Refer2} and references therein. Note the borderline case $(r,q) = (\infty,d)$ is not included in (\ref{serrin2}). This subtle case is of much interest  since we cannot obtain a proof  from usual methods  using the local smallness of certain norms of $u$, which are invariant under the natural scaling
\begin{equation}
\label{natural_scaling}
u(t,x)\rightarrow \lambda
u(\lambda^2t,\lambda x),\quad p(t,x)\rightarrow \lambda^2 p(\lambda^2 t, \lambda x).
\end{equation}
For $d=3$, this case was studied  by Escauriaza, Seregin, and \v{S}ver\'{a}k in a remarkable paper \cite{Seregin1}. The main result of \cite{Seregin1} is the following theorem.

\begin{theorem}[Escauriaza, Seregin, and \v{S}ver\'{a}k]
	\label{ESS} Let $d=3$. Suppose that $u$ is a Leray-Hopf weak solution of the Cauchy problem (\ref{NS1})-(\ref{NS3}) in $(0,T)\times \R^3$ and $u$ satisfies the condition (\ref{serrin1}) with $(r,q) = (\infty,3)$. Then $u\in L_5((0,T)\times\R^3)$, and hence it is smooth and unique in $(0,T)\times\R^3$.
\end{theorem}

Before we explain Theorem \ref{ESS}, we shall recall another important concept involved in the proof, the partial regularity of weak solutions. The study of partial regularity of the Navier-Stokes equations was originated by Scheffer in a series of papers  \cite{Refer23,Refer24,Refer25}. In three space dimensions, he established various partial regularity results for weak solutions satisfying the so-called local energy inequality. Later in a celebrated paper \cite{Caf1}, Caffarelli, Kohn, and Nirenberg first introduced  the notation of {\sl suitable weak solutions}. They called a pair $(u,p)$ a suitable weak solution if $u$ has finite energy norm, $p$ belongs to the Lebesgue space $L_{5/4}$, and $(u,p)$ is  a pair of weak solution to the Navier-Stokes equations and satisfies a local energy inequality. They proved that for any suitable weak solution $(u,p)$,  there exists an open subset in which the velocity field $u$ is H\"{o}lder continuous, and the complement of it has zero 1D Hausdorff measure. In \cite{Refer21}, with zero external force and assuming $p\in L_{3/2}$, Lin gave a more direct and concise  proof for Caffarelli, Kohn, and Nirenberg's result. A detailed treatment was later given by Ladyzhenskaya and Seregin in \cite{Refer18}. Thereafter, Caffarelli, Kohn, and Nirenberg's partial regularity result for the 3D time-dependent Navier-Stokes equations was extended up to the flat boundary by Seregin \cite{Refer23b} and to the $C^2$ boundary by Seregin, Shilkin, and Solonnikov \cite{Refer27b}. The key step in the proofs of partial regularity results is to establish certain $\epsilon$-regularity criteria.  That is, intuitively speaking, if some scale invariant quantities are small then the solution is locally regular. Such results played a crucial part in the proof of \cite{Seregin1}.  For the higher dimensional boundary partial regularity cases, Dong and Gu \cite{Dong2} studied 4D time-dependent and 6D stationary incompressible Navier-Stokes equations. They proved that in both cases, the singular points sets have zero 2D Hausdorff measure up to the boundary. For the 4D time-dependent case, they obtained two boundary $\epsilon$-regularity criteria  \cite[Theorems 1.1 and 1.2]{Dong2}. In Sections \ref{sec_holder_1} and \ref{sec_holder_2}, we will extend  \cite[Theorem 1.2]{Dong2} to higher dimensions assuming we have  certain norms of $u$ and $p$ bounded  and later use those criteria as  tools to prove the main results of this paper.

Back to Theorem \ref{ESS}, the proofs in \cite{Seregin1} are highly nontrivial and rely on certain regularity criteria in the light of \cite{Caf1}, \cite{Refer21}, and \cite{Refer18}. These regularity criteria may break down when the dimension increases, which inspires us to search for a way to modify and generalize the argument.  Another main ingredient of the proof is a backward uniqueness theorem of heat equations with bounded lower order coefficients in the half space (see \cite{Seregin2}). We will also use this part of argument in the proof of our theorems. Under an additional assumption on the pressure, there are some extensions of Theorem \ref{ESS} to the half space case and the bounded domain case; we refer the reader to \cite{Seregin3} and  \cite{Refer19} for some results in this direction. See also \cite{MR2784068, MR3475661, MR3629487, MR3713543, MR3803715, 2018arXiv180203164A} and the references therein for other related results.
As for the extension to the higher dimensional Navier-Stokes equations, in \cite{Dong1}, Dong and Du used Schoen's trick to establish another regularity criterion similar to Theorem \ref{ESS}  and extended this result to $\R^d$ for $d\geq 3$. Unfortunately, there is a gap in the proof of \cite[Lemma 3.2]{Dong1}. We borrow the idea from \cite{Dong1}, that is, to find a priori $L_{\infty}$ bound only depends on the $L^t_{\infty}L^x_d$ norm and the dimension. Instead of using Schoen's trick, we prove two  $\epsilon$-regularity criteria adapted from  \cite[Theorem 1.2]{Dong2} to provide a unified approach to obtain results in the spirit of Theorem \ref{ESS} for  both whole space $\R^d$ and half space $\R^d_+$ for $d\geq 4$.

We now state the main results of the article. The notation in Theorems \ref{thm_main} and \ref{thm_main3} is introduced in Section \ref{prelim}. A remark is that though we state and prove the following theorems for $d\geq 4$, with a minor modification of the exponents in the scale invariant quantities we defined in Section \ref{subsec_scale}, we can give an alternative proof of the case when $d=3$ which has been proved before in \cite{Seregin1} and \cite{Seregin3}.

\begin{theorem}
	\label{thm_main}
Let $d\geq 4$ be an integer.  Let $\Omega=\mathbb{R}^d$ or $\mathbb{R}^d_+:=\{x\in \mathbb{R}^d:x_d>0\}$ and $ T\in(0,\infty]$. Suppose   $(u,p)$ is a pair of Leray-Hopf weak solution to the Cauchy problem in $(0,T)\times \Omega$. If $u$ satisfies the following condition for some $K>0$,
\begin{equation}\label{eqn_Ldcondition}
u \in L_{\infty}^tL_d^x((0,T)\times \Omega), \quad \norm{u}_{L_{\infty}^tL_d^x((0,T)\times \Omega)}\leq K.
\end{equation}
When $\Omega=\mathbb{R}^d_+$, we additionally assume that  $p$ satisfies the local condition: for any $z_0\in (0,T)\times\mathbb{R}_+^d$,
\begin{equation}\label{presure_bd}
\norm{p}_{L_{2-1/d}(Q(z_0,1)\cap (0,T)\times \R^d_+)}\leq K.
\end{equation}
and $u$ satisfies the boundary condition
\begin{equation*}
u(t,x)=0 \quad \text{ on } x_d=0,\  0\leq t\leq T.
\end{equation*}
Then $u \in L_{d+2}((0,T)\times\Omega)$, and hence it is regular up to the boundary in $(0,T)\times \Omega$.
Moreover, if $T=\infty$, we have
	\begin{equation}
	\label{eqn_main_thm2}
	\lim_{t\rightarrow \infty}\norm{u(t,\cdot)}_{L_{\infty}^x(\W)}=0.
	\end{equation}
\end{theorem}
\begin{theorem}
	\label{thm_main3}
	Let $d\geq 4$ be an integer. Suppose $(u,p)$ is a pair of Leray-Hopf weak solution to the Navier-Stokes  problem (\ref{NS1}) in $\R^d_+$. Let $Q^+:=\{z=(t,x):x\in\mathbb{R}^d, \abs{x}<1, x_d>0, -1<t<0\}$. Assume $(u,p)$ satisfies the conditions:
\begin{equation*}
u\in L^t_{\infty}L^x_d(Q^+),\quad p\in L_{2-1/d}(Q^+).
\end{equation*}
and the boundary condition
\begin{equation*}
u(t,x)=0 \quad \text{ on } x_d=0,\ |x|\leq 1,\ -1\leq t\leq 0.
\end{equation*}
Then $u$ is H\"{o}lder continuous in the closure of the set
\begin{equation*}
Q^+(1/4):=\{z=(t,x):x\in\mathbb{R}^d, \abs{x}<1/4, x_d>0,  -(1/4)^2<t<0\}.
\end{equation*}
\end{theorem}

We give a brief description of our argument for the main theorem. By adding conditions (\ref{eqn_Ldcondition})-(\ref{presure_bd}), we extend  \cite[Theorem 1.2]{Dong2} to an $\epsilon$-regularity criterion which reads that if certain scale invariant quantities are small then the solution is locally H\"{o}lder continuous. As in \cite{Seregin1}, we start with proof by contradiction and blow up the solution $(u,p)$ near a singular point at the first blow-up time. We can show the scale invariant quantities are uniformly bounded along a blow-up sequence $(u_k,p_k)$, hence this implies there exists a pair of limiting suitable weak solution $(u_{\infty},p_{\infty})$ to the Navier-Stokes equations. Furthermore, outside of a large cylinder, we can show the scale invariant quantities are indeed uniformly small for all $(u_k,p_k)$'s. Thus we can use the $\epsilon$-regularity criterion to get local H\"{o}lder continuity and uniform local $L_{\infty}$ bound for $u_k$'s. Together with $L_p$-convergence, we can show the local boundedness of  $u_{\infty}$ as well as $u_{\infty}(0,\cdot)=0$ by reversing the blow-up procedure. Then by applying the  backward uniqueness theorem proved in \cite{Seregin2} to the vorticity equation, we can see that $\text{curl }u_{\infty}=0$ in the outside region for all time, which further implies that $u_{\infty}\equiv 0$ by using the spatial analyticity of strong solutions and the weak-strong uniqueness of the Navier-Stokes equations. The rest part of the proof follows the approach in \cite{Dong1}. Utilizing the $\epsilon$-regularity criteria proved in Sections \ref{sec_holder_1} and \ref{sec_holder_2}, we first show there exists a $u_{k_0}$ that is regular around the origin, hence this  contradicts with  the  assumption that $u$ blows up near a singular point. Next we bound the sup norm of $u$ to conclude $u \in L_{d+2}((0,T)\times\Omega)$. For $T=\infty$, a key observation is that $u$ is in $L_4((0,\infty)\times\R^d)$, which implies the smallness of its $L_4$ norm in $(T,\infty)\times\R^d$ for large $T$ and furthermore the smallness of the scale invariant quantities on any cylinder beyond time $T$. Again we can apply the $\epsilon$-criteria to get a uniform $L_{\infty}$ bound on the scaled solutions beyond time $T$. We finally  prove the decay with respect to the time by scaling back to the original $u$.

The remaining part of the paper is organized as follows. We introduce notation and terminologies in Section \ref{prelim}. Sections \ref{sec_holder_1} and \ref{sec_holder_2} are devoted to $\epsilon$-regularity criteria in the whole space and half space, respectively.  We use a three-step approach  to obtain the $\epsilon$-regularity criteria for both the whole space  and the half space. In the first step, we give some estimates of the scale invariant quantities, which are by now standard and essentially follow the arguments in    \cite{Refer21,Refer2b}. In the second step, we establish a  decay estimate of certain scale invariant quantities by using an iteration argument based on the estimates we proved in the first step. In the third step, we apply parabolic regularity to get an  estimate of $L_{2-\frac{1}{d}}$-mean oscillations of $u$, which yields the H\"{o}lder continuity of $u$ according to Campanato's characterization of H\"{o}lder continuous functions.   The main difference between the two cases lies in the treatment of the pressure term. In the interior case, the pressure can be decomposed into a sum of a harmonic function and a term  controlled by $u$ using the Calder\'{o}n-Zygmund estimate. In the boundary case, we need the additional assumption (\ref{presure_bd}) on the pressure to use classical $L_p$ estimates for linear Stokes system to get a more subtle control of the pressure.  In Section \ref{sec_thm1}, we present the proof of Theorem \ref{thm_main} via the blow-up procedure mentioned previously.   Theorem \ref{thm_main3} is another application of the $\epsilon$-regularity criteria we proved in Sections \ref{sec_holder_1} and \ref{sec_holder_2} . We briefly describe the proof of Theorem \ref{thm_main3} in Section \ref{sec_thm2}.
\section{Preliminaries}
Let $\W$ be a domain in $\R^d$ and $\W_T:=(0,T)\times\W$. We denote $\dot{C}_0^{\infty}(\W)$ the space of divergence-free infinitely differentiable vector fields with compact support in $\W$. Let $\dot{J}(\W)$ and $\dot{J}_2^1(\W)$ be the closures of $\dot{C}^{\infty}_0(\W)$ in the spaces $L_2(\W)$ and $W^1_2(\W)$, respectively.
\label{prelim}
\subsection{Leray-Hopf weak solutions.}\label{LH_def}
By a Leray-Hopf weak solution of (\ref{NS1})-(\ref{NS2}) in $\W_T$, we mean a vector field $u$ such that:

	\textsl{i)} $u\in L_{\infty}(0,T;\dot{J}(\W))\cap L_2(0,T;\dot{J}^1_2(\W))$;
	
	\textsl{ii)} the function $t\rightarrow \int_{\W}u(t,x)\cdot w(x)\,dx$ is continuous on $[0,T]$ for any $w\in L_2(\W)$;
	
	\textsl{iii)} the equation (\ref{NS1}) holds weakly in the sense that for any $w\in \dot{C}_0^{\infty}(\W_T)$,
	\begin{equation*}
		\int_{\W_T}(-u\cdot\partial_t w-u\otimes u:\nabla w+\nabla u: \nabla w)\, dx\,dt=0;
	\end{equation*}
	
	\textsl{iv)} The energy inequality:
\begin{equation*}
	\frac{1}{2}\int_{\W}|u(t,x)|^2\, dx+\int_{\W_t}|\nabla u|^2\,dx\,ds \leq \frac{1}{2}\int_{\W}|a(x)|^2\,dx
\end{equation*}
holds for any $t\in[0,T]$, and we have
$$\norm{u(t,\cdot)-a(\cdot)}_{L_2}\rightarrow 0\quad \text{as } t\rightarrow0.$$

When $\W = \R^d$ or $\R^d_+$, for any $a\in\dot{J}(\W)$, there exists at least one Leray-Hopf weak solution to the Cauchy problem (\ref{NS1})-(\ref{NS2}) on $(0,\infty)\times\W$. See \cite{Leray1} and \cite{Hopf1}.

\subsection{Suitable weak solutions}
The definition of suitable weak solutions was introduced in \cite{Caf1} . We say a pair $(u,p)$ is a suitable weak solution of the Navier-Stokes equations on the set $\W_T$ if

	\textsl{i)} $u\in L_{\infty}(0,T;\dot{J}(\W))\cap L_2(0,T;\dot{J}^1_2(\W))$ and $p\in L_{2-1/d}(\W_T)$;

	\textsl{ii)} $u$ and $p$ satisfy equation (\ref{NS1}) in the sense of distribution.

	\textsl{iii)} For any $t\in(0,T)$ and for any nonnegative function $\psi\in C_0^{\infty}(\overline{\W_T})$ vanishing in a neighborhood of the  boundary $\{t=0\}\times \Omega$, the integrals in the following local energy inequality are summable and the inequality holds true:
\begin{align}\label{eqn_sw_energy}
\begin{split}
\text{ess sup}_{0\leq s\leq t}&\int_{\Omega}\abs{u(s,x)}^2\psi(s,x)\,dx+2\int_{\Omega_t}\abs{\nabla u}^2\psi \,dx \,ds\\
&\leq \int_{\Omega_t}\{\abs{u}^2(\psi_t+\Delta\psi)+(\abs{u}^2+2p)u\cdot\nabla\psi\}\,dx \,ds.
\end{split}
\end{align}

\subsection{Scale invariant quantities}
\label{subsec_scale}
In this paper, we write a point in $[0,T]\times\R^d$ as $z=(t,x)=(t,x_1,x_2,...,x_d)=(t,x',x_d)$, where $x'=(x_1,x_2,...,x_{d-1})$.
We shall use the following notation for balls, half balls, spheres, half spheres, parabolic cylinders, half parabolic cylinders, and parabolic boundaries:
\begin{align*}
	&B(\hat{x},r) = \{x\in \mathbb{R}^d \mid \abs{x-\hat{x}}<r\}, \quad B_r = B(r) = B(0,r),\quad B=B(1);\\
	&B^+(\hat{x},r)=\{x\in B(\hat{x},r)\mid x=(x',x_d),x_d>\hat{x}_d\},\\
	&B_r^+ = B^+(r) = B^+(0,r), \quad B^+=B^+(1);\\
	&S(\hat{x},r) = \{x\in \mathbb{R}^d \mid \abs{x-\hat{x}}=r\}, \quad S_r = S(r) = S(0,r),\quad S=S(1);\\
	&S^+(\hat{x},r)=\{x\in S(\hat{x},r)\mid x=(x',x_d),x_d>\hat{x}_d\},\\
	&S^+_r = S^+(r) = S^+(0,r), \quad S^+=S^+(1);\\
	&Q(\hat{z},r)=(\hat{t}-r^2,\hat{t}) \times B(\hat{x},r),\quad Q_r = Q(r)=Q(0,r), \quad Q=Q(1);\\
	&Q^+(\hat{z},r)=(\hat{t}-r^2,\hat{t}) \times B^+(\hat{x},r),\quad Q^+_r = Q^+(r)=Q^+(0,r), \quad Q^+=Q^+(1);\\
	&\partial_pQ(\hat{z},r) = [\hat{t}-r^2,\hat{t}) \times S(\hat{x},r)\cup \{t=\hat{t}-r^2\}\times B(\hat{x},r)
\end{align*}
where $\hat{z}=(\hat{t},\hat{x})$ and $\hat{x}_d$ is the $d$-th coordinate of $\hat{x}$.

For the remaining part of the paper, we restrict our discussion to the following domains (except for the local problem in Section \ref{sec_thm2}):
$$\Omega=\mathbb{R}^d \text{ or } \mathbb{R}^d_+,\quad \W_T= (0,T)\times \W,$$
$$\Omega(\hat{x},r)=B(\hat{x},r)\cap \Omega, \quad\omega(\hat{z},r)=Q(\hat{z},r)\cap \W_T.$$
In particular, we denote
$ \R^{d+1}_T = (0,T)\times \R^d$.

We denote mean values of summable functions as follows:
$$[u]_{\hat{x},r}(t)=\frac{1}{|\Omega(\hat{x},r)|}\int_{\Omega(\hat{x},r)}u(t,x)\, dx,$$
$$(u)_{\hat{z},r}=\frac{1}{|\omega(\hat{z},r)|}\int_{\omega(\hat{z},r)}u(z)\,dz,$$
where $\abs{A}$ as usual denotes the Lebesgue measure of the set A.

Now we introduce the following important quantities:
\begin{enumerate}
	\item[i)] When $\Omega = \mathbb{R}^d$,
	\begin{align*}
		&A(r,z_0)=\esssup_{t_0-r^2\le t\le t_0}\frac{1}{r^{d-2}}\int_{B(x_0,r)}\vert u\vert ^2  \,dx,\\
		&E(r,z_0)=\frac{1}{r^{d-2}}\int_{Q(z_0,r)}\vert \nabla u\vert ^2 \,dz,\\
		&C(r,z_0)=\frac{1}{r^{d+2/d-2}}\int_{Q(z_0,r)}\vert u \vert^{2(2d-1)/d} \,dz,\\
		&D(r,z_0)=\frac{1}{r^{d+2/d-2}}\int_{Q(z_0,r)}\vert p-[p]_{x_0,r}(t) \vert^{(2d-1)/d} \,dz.
	\end{align*}
	\item[ii)] When $\Omega = \mathbb{R}^d_+$ or $Q^+$,
	\begin{align*}
		&A^+(r,z_0)=\esssup_{t_0-r^2\le t\le t_0}\frac{1}{r^{d-2}}\int_{\Omega(x_0,r)}\vert u\vert ^2  \,dx,\\
		&E^+(r,z_0)=\frac{1}{r^{d-2}}\int_{\omega(z_0,r)}\vert \nabla u\vert ^2 \,dz,\\
		&C^+(r,z_0)=\frac{1}{r^{d+2/d-2}}\int_{\omega(z_0,r)}\vert u \vert^{2(2d-1)/d} \,dz,\\
		&D^+(r,z_0)=\frac{1}{r^{d+2/d-2}}\int_{\omega(z_0,r)}\vert p-[p]_{x_0,r}(t) \vert^{(2d-1)/d} \,dz.
	\end{align*}
\end{enumerate}
We notice that these quantities are all invariant under the natural scaling \eqref{natural_scaling}.

 In the later part of the paper,  we use notation  $A^{(+)}$ to represent either $A$ or $A^+$ depending on $\Omega=\mathbb{R}^d$ or $\Omega=\mathbb{R}^d_+$  when there is no confusion and similarly for $E^{(+)}, C^{(+)}, D^{(+)}, Q^{(+)}$, etc. We omit $z_0$, the argument for center, from the above expressions and write $A(r)$, $B(r)$, $C(r)$, and $D(r)$ when there is no ambiguity.

\subsection{Strong solutions and spatial analyticity.}
We recall the following local solvability of (\ref{NS1})-(\ref{NS2}) (see, for instance, \cite{Kato1,Giga1,Taylor1,Koch1,local_sol1}), and spatial analyticity of strong solutions (see, for instance, \cite{Giga2,Dong3,analyticity,analyticity1}).
\begin{prop}
	\label{prop_StrongSolva}
	For any divergence-free initial data $a\in L_p(\W)$ with $p\geq d$, where $\W = \mathbb{R}^d$ or $\mathbb{R}_+^d$, the Cauchy problem (\ref{NS1})-(\ref{NS2}) has a unique strong solution $u\in C([0,\delta);L_p(\W))$ for some $\delta>0$. Moreover, $u$ is infinitely differentiable and spatial analytic for $t\in (0,\delta)$.
\end{prop}

In the following two sections, we will show the H\"{o}lder continuity of $u$ given the scale invariant quantities defined previously are sufficiently small. The main difference between the interior estimate and the boundary estimate results from the different estimates of quantities $D$ and $D^+$.
\section{H\"{o}lder Continuity Interior  Estimate} \label{sec_holder_1}
In this section, we consider $\W = \bR^d$. We take a three-step approach to prove the following $\epsilon$-regularity criterion in the whole space.
\begin{theorem}
	\label{thm_holder_whole}
	Let $(u,p)$  be a suitable weak solution of (\ref{NS1})-(\ref{NS2}) in $(0,T)\times\mathbb{R}^{d}$ satisfying (\ref{eqn_Ldcondition}). There exists a universal constant $\epsilon_0$ satisfying the following property. Assume that for a point $z_0\in\mathbb{R}_T^{d+1}$ we have
	$$A(\rho_0,z_0)+E(\rho_0,z_0)+D(\rho_0,z_0)\leq \epsilon_0$$
	for some small $\rho_0>0$.
	Then $u$ is H\"{o}lder  continuous near $z_0$.
\end{theorem}
\subsection{Step 1} We present several  inequalities of the scale invariant quantities.
We will make use of the following interpolation inequality from \cite[Lemma 2.1]{Dong1} substantially.
\begin{lemma}\label{lem2}
For any function $u\in W_2^1(B_r)$, $r>0$, and $q\in [2,2d/(d-2)]$, we have
\begin{align}\label{eqn1}
\begin{split}
\int_{B_r}|u|^q \,dx \leq &N(q,d)\Big[\left(\int_{B_r}\abs{\nabla u}^2\,dx\right)^{ \frac{d}{2}\left(\frac{q}{2}-1\right)}\left(\int_{B_r}\abs{u}^2\,dx\right)^{\frac{q}{2}-\frac{d}{2}\left(\frac{q}{2}-1\right)}\\
&+r^{-d\left(\frac{q}{2}-1\right)}\left(\int_{B_r}\abs{u}^2\,dx\right)^{\frac{q}{2}}\Big].
\end{split}
\end{align}
\end{lemma}

\begin{proof} Without loss of generality, we assume $r=1$.
	For $q\in [2,2d/(d-2)]$, we use H\"{o}lder's inequality inside the unit ball  $B$,
	\begin{equation*}
		\norm{u}_{L_q(B)}\leq \norm{u}_{L_2(B)}^{1-\frac{d}{q}(\frac{q}{2}-1)}\norm{u}_{L_{\frac{2d}{d-2}}(B)}^{\frac{d}{q}(\frac{q}{2}-1)},
	\end{equation*}
which together with Sobolev embedding theorem  gives
\begin{equation*}
	\int_B\abs{u}^q\, dx\leq
	N\left( \int_B\abs{u}^2\, dx\right)^{\frac{q}{2}-d(\frac{q}{4}-\frac{1}{2})}\left[\left(\int_B\abs{\nabla u}^2\, dx\right)^{d(\frac{q}{4}-\frac{1}{2})}+\left(\int_B\abs{u}^2\, dx\right)^{d(\frac{q}{4}-\frac{1}{2})}\right].
\end{equation*}
The lemma is proved.
\end{proof}


The next lemma is an application of the interpolation inequality proved in Lemma \ref{lem2}.
\begin{lemma}\label{lem7}
For $\alpha \in[0,1]$ and $p\in [2+\frac{4\alpha}{d},d+(4-d)\alpha]$, suppose $r>0$, $Q(z_0,r)\subset\mathbb{R}_T^{d+1}$, and u satisfies the condition (\ref{eqn_Ldcondition}). Then we have
\begin{align*}
\frac{1}{r^{d+2-p}}\int_{Q(z_0,r)}|u|^p\, dz\leq N\left(A(r,z_0)+E(r,z_0)\right)^{\frac{d-p+2\alpha}{d-2}}.
\end{align*}
In particular, taking $\alpha=1$ and  $p = 4-2/d$  we have
$$C(r,z_0) \leq N\left(A(r,z_0)+E(r,z_0)\right)^{\frac{d-2+2/d}{d-2}}.$$
\end{lemma}

\begin{proof}
Without loss of generality,  we assume $r=1$. Because $p\in [2+\frac{4\alpha}{d},d+(4-d)\alpha]$, we know $\frac{pd-2d-4\alpha}{d(d-2)}\geq 0$ and $q:=\frac{2d(d-p+2\alpha)}{d^2-dp+4\alpha}\in \left[2,\frac{2d}{d-2}\right]$. By using H\"{o}lder's inequality and  (\ref{eqn1}) with this $q$, we have
\begin{align*}
\quad \int_{B}|u|^p\, dx
&\leq \left( \int_{B} |u|^{\frac{2d(d-p+2\alpha)}{d^2-dp+4\alpha}}\, dx\right)^{\frac{d^2-dp+4\alpha}{d(d-2)}}\left(\int_{B}|u|^d\,dx\right)^{\frac{pd-2d-4\alpha}{d(d-2)}}\\
&\leq N\left(\left(\int_{B}|\nabla u|^2\, dx\right)^{\alpha}\left(\int_{B}|u|^2\, dx\right)^{\frac{(1-\alpha)d+4\alpha-p}{d-2}}
+\left(\int_{B}|u|^2\,dx\right)^{\frac{d-p+2\alpha}{d-2}}\right),
\end{align*}
where we used (\ref{eqn_Ldcondition}) in the last inequality.
Integrating in time yields the desired result.
\end{proof}
\begin{lemma}
	\label{lemAECD}
	Let $(u,p)$ be a pair of suitable weak solution of (\ref{NS1}). For $\rho>0$ and  $Q(z_0,\rho)\subset \mathbb{R}_T^{d+1}$, we have
	$$A(\rho/2)+E(\rho/2)\leq N\left(C(\rho)^{\frac{d}{2d-1}}+C(\rho)^{\frac{3d}{2(2d-1)}}+ C(\rho)^{\frac{d}{2(2d-1)}}D(\rho)^{\frac{d}{2d-1}}\right),$$
	where $N$ is independent of $z_0$ and $\rho$.
\end{lemma}
\begin{proof}
	By a scaling argument, we may assume $\rho=1$. In the energy inequality (\ref{eqn_sw_energy}), we put $t=t_0$ and choose a suitable smooth cut-off function $\psi$ such that
	$$\psi \equiv 0\text{ in } \mathbb{R}_{t_0}^{d+1}\setminus Q(z_0,1), \quad 0\leq \psi\leq 1\text{ in }\mathbb{R}_{t_0}^{d+1},$$
	$$\psi \equiv 1 \text{ in } Q(z_0,1/2), \quad \abs{\nabla \psi}+\abs{\partial_t \psi}+\abs{\nabla^2\psi}<N \text{ in } \mathbb{R}_{t_0}^{d+1}.$$
	By using (\ref{eqn_sw_energy}), we get
	$$A(1/2)+2E(1/2)\leq N\int_{Q(z_0,1)}\abs{u}^2\, dz+N\int_{Q(z_0,1)}(\abs{u}^2+2\abs{p})\abs{u}\, dz.$$
	Due to H\"{o}lder's inequality, we can obtain
$$
\int_{Q(z_0,1)}\abs{u}^2\, dz\leq N(C(1))^{\frac{d}{2d-1}},\quad \int_{Q(z_0,1)}\abs{u}^3\, dz\leq N(C(1))^{\frac{3d}{2(2d-1)}},
$$
	and
	\begin{align*}
		&\int_{Q(z_0,1)}\abs{p}\abs{u}\, dz\\
		&\leq \left( \int_{Q(z_0,1)}\abs{u}^{\frac{2d-1}{d-1}}\, dz\right)^\frac{d-1}{2d-1}\left(\int_{Q(z_0,1)}\abs{p}^{\frac{2d-1}{d}}\,dz \right)^{\frac{d}{2d-1}}\\
		&\leq C(1)^{\frac{d}{2(2d-1)}}D(1)^{\frac{d}{2d-1}}.
	\end{align*}
	The conclusion of Lemma \ref{lemAECD} follows immediately.
\end{proof}
\begin{lemma}
	\label{lem8}
	Let $(u,p)$ be a pair of suitable weak solution of (\ref{NS1}) when $d\geq 4$. For constants $\gamma \in (0,1/2]$, $\rho >0$,  and $Q(z_0,\rho)\subset \mathbb{R}_T^{d+1}$, we have
\begin{align}\label{eqn2}
 A(\gamma\rho)+E(\gamma\rho)
\leq N\left[\gamma^{2}A(\rho)+\gamma^{-d+1}\left((A(\rho)+E(\rho))^{\frac{d-1}{d-2}}+(A(\rho)+E(\rho))^{\frac{1}{2}}D(\rho)^{\frac{d}{2d-1}}\right)\right].
\end{align}
\end{lemma}
\begin{proof}
Assume $\rho=1$.
Define the backward heat kernel as
$$\Gamma(t,x)=\frac{1}{(4\pi(\gamma^2+t_0-t))^{d/2}}e^{-\frac{\abs{x-x_0}^2}{2(\gamma^2+t_0-t)}},$$
In the energy inequality (\ref{eqn_sw_energy}), we choose $\psi=\Gamma\phi$, where $\phi\in C^{\infty}_0( (t_0-1,t_0+1)\times B(x_0,1))$ is a suitable smooth cut-off function satisfying
$$ 0\leq \phi \leq 1 \quad \text{in } \R \times\mathbb{R}^d , \quad \phi \equiv 1 \quad \text{ in  } Q(z_0,1/2),$$
$$|\nabla \phi |\leq  N,\quad |\nabla^2\phi| \leq N, \quad|\partial_t\phi |\leq N \text{ in  }  \R \times\mathbb{R}^d.$$
By using the equation
$$\Delta \Gamma +\Gamma_t=0,$$
we have	
\begin{align}\label{eqn8}
\begin{split}
\text{ess sup}_{0\leq s\leq t}&\int_{B(x_0,1)}\abs{u(t,x)}^2\Gamma(t,x)\phi(t,x)\,dx+2\int_{Q(z_0,1)}\abs{\nabla u}^2\Gamma\phi \,dz\\
&\leq \int_{Q(z_0,1)}\{\abs{u}^2(\Gamma\phi_t+\Gamma\Delta\phi+2\nabla\phi\nabla\Gamma)\\
&\quad +(\abs{u}^2+2\abs{p-[p]_{x_0,1}})u\cdot(\Gamma\nabla\phi+\phi\nabla\Gamma \} \,dz.
\end{split}
\end{align}
The test function has the following properties:
\begin{enumerate}
\item[(i)] For some constant $c>0$, on $Q(z_0,\gamma)$ it holds that
$$\Gamma\phi = \Gamma \geq c\gamma^{-d}.$$
\item[(ii)]For any $z\in Q(z_0,1)$, we have
$$\abs{\Gamma(z)\phi(z)}\leq N\gamma^{-d}, \quad \abs{\nabla\Gamma(z)\phi(z)}+\abs{\Gamma(z)\nabla\phi(z)}\leq N\gamma^{-d-1}.$$
\item[(iii)] For any $z\in Q(z_0,1)$, we have
$$\abs{\Gamma(z)\phi_t(z)}+\abs{\Gamma(z)\Delta\phi(z)}+\abs{\nabla\Gamma(z)\nabla\phi(z)}\leq N.$$
\end{enumerate}
Therefore (\ref{eqn8}) yields
 \begin{align*}
\begin{split}
A(\gamma)+E(\gamma)&=
\gamma^{-d+2}\text{ess sup}_{0\leq s\leq t}\int_{B(x_0,\gamma)}\abs{u(t,x)}^2\,dx+\gamma^{-d+2}\int_{Q(z_0,\gamma)}\abs{\nabla u}^2 \,dz\\
&\leq N\left[\gamma^2\int_{Q(z_0,1)}\abs{u}^2 +\gamma^{-d+1} \int_{Q(z_0,1)}(\abs{u}^2+2\abs{p-[p]_{x_0,1}})|u|\,dz\right].
\end{split}
\end{align*}
Recall that $d\ge 4$. Applying Lemma \ref{lem7} with $\alpha=1$ and  $p=3$ we have
\begin{align*}
\int_{Q(z_0,1)}|u|^3 \,dz \leq N\left(A(1,z_0)+E(1,z_0)\right)^{\frac{d-1 }{d-2}},
\end{align*}
and again applying  Lemma \ref{lem7} with $\alpha = \frac{d}{4(d-1)}$ and  $p=\frac{2d-1}{d-1}$, we have
\begin{align*}
& \int_{Q(z_0,1)} |p-[p]_{x_0,1}| |u|\, dz\\
&\leq \left(\int_{Q(z_0,1)} |u|^{\frac{2d-1}{d-1}}\, dz\right)^{\frac{d-1}{2d-1}} \left(\int_{Q(z_0,1)} |p-[p]_{x_0,1}|^{\frac{2d-1}{d}}\, dz \right)^{\frac{d}{2d-1}}\\
& \leq N\left[\left(A(1,z_0)+E(1,z_0)\right)^{\frac{1}{2}}D(1,z_0)^{\frac{d}{2d-1}}\right].
\end{align*}
The lemma is proved.
\end{proof}

\begin{lemma}\label{lem_CD} Let $(u,p)$ be a pair of suitable weak solution of (\ref{NS1}). For constants $\gamma \in (0,1/2]$, $\rho >0$,  and $Q(z_0,\rho)\subset \mathbb{R}_T^{d+1}$, we have
\begin{equation}\label{eqn20}
D(\gamma\rho)\leq N(d)[\gamma^{-d-2/d+2}C(\rho)+\gamma^{4-3/d}D(\rho)].
\end{equation}
\end{lemma}
\begin{proof}
Let $r=\gamma \rho$ and $\eta(x)$ be a smooth cut-off function supported in $B(1)$, $0\leq \eta\leq 1$ and $\eta\equiv 1$ on $\bar{B}(2/3)$. In the sense of distribution, for a.e. $t\in(t_0-\rho^2,t_0)$, one has
$$\Delta p=D_{ij}(u_iu_j).$$
We consider the decomposition
$$p=p_{x_0,\rho}+h_{x_0,\rho},$$
where $p_{x_0,\rho} $ is the Newtonian potential of
$$D_{ij}(u_iu_j\eta((x-x_0)/\rho)).$$
Then $h_{x_0,\rho}$ is harmonic in $B(x_0,2\rho/3)$.

By using the Calder\'{o}n-Zygmund estimate, we have
\begin{equation}\label{eqn21}
\int_{Q(z_0,r)}|p_{x_0,\rho}|^{\frac{2d-1}{d}} \,dz\leq \int_{Q(z_0,\rho)}|p_{x_0,\rho}|^{\frac{2d-1}{d}} \,dz \leq N\int_{Q(z_0,\rho)}|u|^{\frac{2(2d-1)}{d}} \,dz.
\end{equation}
From the Poincar\'{e} inequality and the fact that any Sobolev norm of harmonic function $h_{x_0,\rho}-[h_{x_0,\rho}]_{x_0,r}$ in a smaller ball can be estimated by any of its $L_p$ norm in $B(x_0,2\rho/3)$, one obtains
\begin{align}\label{eqn22}
\begin{split}
& \int_{B(x_0,r)}|h_{x_0,\rho}-[h_{x_0,\rho}]_{x_0,r}|^{\frac{2d-1}{d}} \,dx\\
&\leq Nr^{\frac{2d-1}{d}}\int_{B(x_0,r)}|\nabla h_{x_0,\rho}|^{\frac{2d-1}{d}} \,dx\\
&\leq Nr^{d+\frac{2d-1}{d}}\sup_{B(x_0,r)}|\nabla h_{x_0,\rho}|^{\frac{2d-1}{d}}\\
&\leq N\left(\frac{r}{\rho}\right)^{d+\frac{2d-1}{d}}\int_{B(x_0,\rho)}|h_{x_0,\rho}-[p]_{x_0,\rho}|^{\frac{2d-1}{d}}\,dx.
\end{split}
\end{align}
Integrating (\ref{eqn22}) in $t\in(t_0-r^2,t_0)$, we obtain
\begin{align}\label{eqn_48}
	\begin{split}
		 &\int_{Q(z_0,r)}|h_{x_0,\rho}-[h_{x_0,\rho}]_{x_0,r}|^{\frac{2d-1}{d}} \,dz\\
&\leq N\left(\frac{r}{\rho}\right)^{d+\frac{2d-1}{d}}\int_{Q(z_0,\rho)}\left(|p-[p]_{x_0,\rho}|^{\frac{2d-1}{d}}+| p_{x_0,\rho}|^{\frac{2d-1}{d}}\right) \,dz\\
&\leq N\left(\frac{r}{\rho}\right)^{d+\frac{2d-1}{d}}\int_{Q(z_0,\rho)}\left(|p-[p]_{x_0,\rho}|^{\frac{2d-1}{d}}+| u|^{\frac{2(2d-1)}{d}}\right) \,dz,
\end{split}
\end{align}
where we used (\ref{eqn21}) in the last inequality. We combine (\ref{eqn_48}), (\ref{eqn21}), and use the triangle inequality to have
\begin{align*}
	&\int_{ Q(z_0,r)}|p-[p]_{x_0,r}|^{\frac{2d-1}{d}}\,dz \\
	& \leq N \int_{ Q(z_0,r)}|p-[h_{x_0,\rho}]_{x_0,r}|^{\frac{2d-1}{d}}\,dz \\
	& \leq N \int_{ Q(z_0,r)}\left(|h_{x_0,\rho}-[h_{x_0,\rho}]_{x_0,r}|^{\frac{2d-1}{d}}+ |p_{x_0,\rho}|^{\frac{2d-1}{d}}\right)\,dz \\
	& \leq N\left(\frac{r}{\rho}\right)^{d+\frac{2d-1}{d}}\int_{Q(z_0,\rho)}\left(|p-[p]_{x_0,\rho}|^{\frac{2d-1}{d}}+| u|^{\frac{2(2d-1)}{d}}\right) \,dz.
\end{align*}
The lemma is proved.
\end{proof}
\begin{corollary}
	\label{col_CD}
	 Let $(u,p)$ be a pair of suitable weak solution of (\ref{NS1}). For some $z_0\in \R^d_T$, suppose there exist $\rho_0>0$ and $C_1>0$, such that $Q(z_0,\rho_0)\subset \mathbb{R}_T^{d+1}$ and $C(\rho,z_0)\leq C_1$ for all $\rho\in(0,\rho_0]$ and $D(\rho_0,z_0)\leq C_1$. Then we can find $C=C(C_1,d)>0$ such that $D(\rho,z)\leq C$ for all  $z\in Q(z_0,1/2)$ and $\rho\in(0,\rho_0/2]$.
\end{corollary}
\begin{proof}
	Note that $Q(z,\rho_0/2)\subset Q(z_0,\rho_0)$ for $z\in Q(z_0,1/2)$, hence $D(\rho_0/2,z)\leq ND(\rho_0,z_0)$.
	By fixing $\gamma$ small enough that $N\gamma^{4-3/d}\leq 1/2$ and using (\ref{eqn20}), we have
	$$D(\gamma\rho_0/2,z)\leq \frac{1}{2}D(\rho_0/2,z)+C_2,$$
	where $C_2 = N\gamma^{-d-2/d+2}C_1$.
	Inductively, for any integer $k$ we have
	$$D(\gamma^k\rho_0/2,z)\leq \frac{1}{2^k}D(\rho_0/2,z)+C_2\sum_{j=0}^{k-1}\frac{1}{2^j}\leq C_1+2C_2:=C_3.$$
	Now for any $\gamma^{k+1}\rho_0\leq \rho <\gamma^k\rho_0$, we can control $D(\rho/2,z)$ by
	$$D(\rho/2,z)\leq \left(\frac{\gamma^k\rho_0}{\gamma^{k+1}\rho_0}\right)^{d+2/d-2}D(\gamma^k\rho_0/2,z)\leq \gamma^{-d-2/d+2}C_3:=C.$$
	The corollary is proved because $C$ is independent of $k$.
	\end{proof}

\subsection{Step 2}
We will find some decay rates for the scale invariant quantities with respect to the radius assuming the quantities are initially small.
\begin{lemma}\label{lem_a0_whole}
There exists a universal constant $\epsilon_0>0$ satisfying the following property. Suppose that for some $z_0=(t_0,x_0)$ and $\rho_0>0$, it holds that $Q(z_0,\rho_0)\subset \mathbb{R}_T^{d+1}$ and
\begin{equation}\label{eqn5}
A(\rho_0,z_0)+E(\rho_0,z_0)+D(\rho_0,z_0)\leq \epsilon_0.
\end{equation}
Then fixing any  $\alpha_0\in (0,2) $, there exists $N>0$ such that for any $\rho
\in(0,\rho_0/2)$ and $z\in Q(z_0,\rho_0/2)$, the following estimate holds uniformly
\begin{equation}\label{eqn6}
A(\rho,z)+E(\rho,z)+C(\rho,z)^{\frac{d-2}{d-2+2/d}}+D(\rho,z)\leq N\epsilon_0^{\alpha_0/2}\left(\frac{\rho}{\rho_0}\right)^{\alpha_0},
\end{equation}
where  $N$ is a positive constant depending on $\alpha_0$, but independent of $\epsilon_0$, $\rho_0$, $\rho$, and $z$.
\end{lemma}

\begin{proof}
For any $z\in Q(z_0,\rho_0/2)$, by (\ref{eqn5}) and
$$Q(z,\rho_0/2)\subset Q(z_0,\rho_0)\subset \mathbb{R}_T^{d+1},$$
we get
\begin{equation}
\label{eqn_step2}
	A(\rho_1,z)+E(\rho_1,z)+D(\rho_1,z)\leq N\epsilon_0,
\end{equation}
where $\rho_1=\rho_0/2.$
By Lemma \ref{lem7},
\begin{equation}
\label{eqn316}
		C(\rho_1,z)\leq(N\epsilon_0)^{\frac{d-2+2/d}{d-2}}.
\end{equation}

Next we fix an auxiliary parameter  $\alpha\in(\alpha_0,2)$. By a scaling argument, we first discuss a special case when $\rho_1^{\alpha} = N\epsilon_0 < 1 $.
In this case, we can prove the following decay rates inductively:
\begin{equation}\label{eqn13}
A(\rho_k)+E(\rho_k)\leq \rho_k^{\alpha},\quad C(\rho_k)^{\frac{d-2}{d-2+2/d}}\leq \rho_k^{\alpha}, \quad D(\rho_k)\leq \rho_k^{\alpha},
\end{equation}
where $\rho_{k+1}=\rho_k^{1+\beta}=\rho_1^{(1+\beta)^k}$ and $\beta$ is a small number to be specified. For $k=1$, the statement follows from  (\ref{eqn_step2}), (\ref{eqn316}), and our assumption that  $\rho_1^{\alpha} = N\epsilon_0$. Next  by choosing $\gamma=\rho_k^{\beta}$ and $\rho=\rho_k$ in (\ref{eqn2}) and (\ref{eqn20}), we have
$$
A(\rho_{k+1})+E(\rho_{k+1})\leq N\left[ \rho_k^{2\beta+\alpha}+\rho_k^{(-d+1)\beta}\left(\rho_k^{\frac{d-1}{d-2}\alpha}+\rho_k^{\left(\frac{1}{2}+\frac{d}{2d-1}\right)\alpha}\right)\right].
$$
$$D(\rho_{k+1})\leq N\left[ \rho_k^{(-d-\frac{2}{d}+2)\beta+\left(1+\frac{2}{d(d-2)}\right)\alpha}+\rho_k^{\left(4-\frac{3}{d}\right)\beta+\alpha}\right].$$
We choose $\beta$ satisfying
$$\beta < \min\left\{\frac{\alpha}{(d-2)(d+\alpha-1)},\frac{\alpha}{2(2d-1)(d+\alpha-1)},\frac{2\alpha}{d(\alpha+d+\frac{2}{d}-2)(2d-2)} \right\}. $$
Then all the exponents on the right-hand sides are greater than $(1+\beta)\alpha$.

Now we can find $\xi>0$ depending on $\beta$ such that
 \begin{equation*}
A(\rho_{k+1})+E(\rho_{k+1})\leq N\rho_{k+1}^{\alpha+\xi}, \quad D(\rho_{k+1})\leq N\rho_{k+1}^{\alpha+\xi},
\end{equation*}
where $N$ is a constant independent of $k$ and $\xi$. By taking $\epsilon_0$ small enough such that $N\rho_{k+1}^{\xi}< N\rho_1^{\xi} <N(N\epsilon_0)^{\xi/2}<1$ , we obtain
\begin{equation*}
A(\rho_{k+1})+E(\rho_{k+1})\leq \rho_{k+1}^{\alpha}, \quad D(\rho_{k+1})\leq \rho_{k+1}^{\alpha}.
\end{equation*}
By induction, we have justified (\ref{eqn13}) for the case when $\rho_1^{\alpha} = N\epsilon_0$.

For convenience, we additionally assume that the parameter $\beta$ satisfying
$$
\alpha_0 < \min\left\{\frac{1}{1+\beta}(\alpha-(d-2)\beta),\frac{1}{(1+\beta)}\left(\alpha-\left(d+\frac{2}{d}-2\right)\beta\right)\right\}.$$ There always exists feasible $\beta$ because $\alpha>\alpha_0$.

Now for any $\rho\in(0,\rho_0/2)$, we can find a positive integer $k$ such that $\rho_{k+1}\leq \rho<\rho_k$. Then
$$A(\rho)+E(\rho)\leq \left(\frac{\rho_k}{\rho_{k+1}}\right)^{d-2}(A(\rho_k)+E(\rho_k))\leq \rho_k^{\alpha-(d-2)\beta}=\rho_{k+1}^{\frac{1}{1+\beta}(\alpha-(d-2)\beta)}\leq \rho^{\alpha_0},$$
$$D(\rho)\leq \left(\frac{\rho_k}{\rho_{k+1}}\right)^{d+\frac{2}{d}-2}D(\rho_k)\leq \rho_k^{\alpha-(d+\frac{2}{d}-2)\beta}=  \rho_{k+1}^{\frac{1}{1+\beta}(\alpha-(d+\frac{2}{d}-2)\beta)}\leq \rho^{\alpha_0}.$$
Hence we have proved \eqref{eqn6} when $\rho_1^{\alpha} = N\epsilon_0$. For the general case, we  use the scale invariant property of the quantities and apply the previous results with an additional scaling factor $N\epsilon_0^{\alpha_0/\alpha}\rho_{0}^{-\alpha_0}$   on the right-hand side.
\end{proof}

\subsection{Step 3} In this final step, we first use parabolic $L_p$ estimates to further improve the decay rate and then conclude the result by using Campanoto's characterization of H\"{o}lder continuity.
\begin{lemma}\label{lem1} Suppose $f(\rho_{0})\leq C_0$. If there exist $\alpha >\beta >0$ and $C_1,C_2>0$ such that for any $0<r<\rho\le\rho_{0}$, it holds that
$$f(r)\leq C_1\left(\frac{r}{\rho}\right)^{\alpha}f(\rho)+C_2\rho^{\beta},$$
then there exist constants $ C_3,C_4>0$ depending on  $C_0,C_1,C_2,\alpha,\beta$, such that
$$f(r)\leq C_3\left(\frac{r}{\rho_0}\right)^{\beta}f(\rho_0)+ C_4r^{\beta}$$
for $0<r\le\rho_{0}$.

\end{lemma}
\begin{proof}
See, for instance, \cite[Chapter III, Lemma 2.1]{iter}.
\end{proof}
By Lemma \ref{lem_a0_whole} we know, for any  small $\delta_1>0$, the following estimates are true for all $\rho\in (0,\rho_0/2)$ sufficiently small and $z_1:=(t_1,x_1)\in Q(z_0,\rho_0/2)$:
\begin{align}\label{eqn15}
&\int_{Q(z_1,\rho)}|u|^{4-\frac{2}{d}}\, dz\leq N\rho^{d+\frac{2}{d-2}-\delta_1},\\
\label{eqn16}
&\int_{Q(z_1,\rho)}|p-[p]_{x_1,\rho}|^{2-\frac{1}{d}}\, dz\leq N\rho^{d+\frac{2}{d}-\delta_1}.
\end{align}
Let $v$ be the unique weak solution to the heat equation
$$\partial_t v-\Delta v=0 \quad \text{in }  Q(z_1,\rho),$$
with the boundary condition $v=u$ on $\partial_pQ(z_1,\rho)$.
Let $0<r<\rho$. By the Poincar\'{e} inequality with zero mean value  and using the fact that $L_{\infty}$ norm of the gradient of a caloric function in a smaller cylinder is controlled by any $L_p$ norm of it in a larger cylinder. We have

\begin{align}\label{eqn18}
\begin{split}
\int_{Q(z_1,r)}|v-(v)_{z_1,r}|^{2-\frac{1}{d}}\, dz&=\int_{Q(z_1,r)}|v-(u)_{z_1,\rho}-(v-(u)_{z_1,\rho})_{z_1,r}|^{2-\frac{1}{d}}\, dz  \\
&\leq r^{d+4-\frac{1}{d}}\norm{\nabla \left(v-(u)_{z_1,\rho}\right)}^{2-\frac 1 d}_{L_{\infty}(Q(z_1,r))}\\
&\leq N\left(\frac{r}{\rho}\right)^{d+4-\frac{1}{d}}\int_{Q(z_1,\rho)}|v-(u)_{z_1,\rho}|^{2-\frac{1}{d}}\, dz.
\end{split}
\end{align}
Denote $w=u-v$. Then $w$ satisfies the inhomogeneous heat equation
$$\partial_t w_i-\Delta w_i=-\partial_j(u_iu_j)-\partial_i(p-[p]_{x_1,\rho}) \quad \text{in } Q(z_1,\rho)$$
with the zero boundary condition. By the classical $L_p$ estimate for the heat equation, we have
$$\Vert\nabla w\Vert_{L_{2-\frac{1}{d}}(Q(z_1,\rho))}\leq N\left[\left\lVert |u|^2 \right\rVert_{L_{2-\frac{1}{d}}(Q(z_1,\rho))} + \left\lVert p-[p]_{x_1,\rho}\right\rVert_{L_{2-\frac{1}{d}}(Q(z_1,\rho))}  \right],$$
which together with (\ref{eqn15}) and (\ref{eqn16})  yields
\begin{equation}\label{eqn17}
\int_{Q(z_1,\rho)}|\nabla w|^{2-\frac{1}{d}}\, dz\leq N \rho^{d+\frac{2}{d}-\delta_1}.
\end{equation}
By the Poincar\'{e} inequality with zero boundary condition, we get from \eqref{eqn17} that
\begin{equation}\label{eqn19}
\int_{Q(z_1,\rho)}|w|^{2-\frac{1}{d}}\, dz\leq N \rho^{d+\frac{2}{d}-\delta_1+2-\frac{1}{d}} = N \rho^{d+2+\frac{1}{d}-\delta_1}.
\end{equation}
Using (\ref{eqn18}), (\ref{eqn19}), and the triangle inequality, we have
\begin{align*}
&\int_{Q(z_1,r)}|u-(u)_{z_1,r}|^{2-\frac{1}{d}}\, dz \\
&\leq \int_{Q(z_1,r)}|v-(v)_{z_1,r}|^{2-\frac{1}{d}}\, dz+\int_{Q(z_1,r)}|w-(w)_{z_1,r}|^{2-\frac{1}{d}}\, dz\\
& \leq N\left(\frac{r}{\rho}\right)^{d+4-\frac{1}{d}}\int_{Q(z_1,\rho)}|v-(u)_{z_1,\rho}|^{2-\frac{1}{d}}\, dz + N \rho^{d+2+\frac{1}{d}-\delta_1}\\
&\leq N\left(\frac{r}{\rho}\right)^{d+4-\frac{1}{d}}\int_{Q(z_1,\rho)}|u-(u)_{z_1,\rho}|^{2-\frac{1}{d}}\, dz  \\
&\quad +N\left(\frac{r}{\rho}\right)^{d+4-\frac{1}{d}}\int_{Q(z_1,\rho)}|w|^{2-\frac{1}{d}}\, dz + N \rho^{d+2+\frac{1}{d}-\delta_1}\\
& \leq N\left(\frac{r}{\rho}\right)^{d+4-\frac{1}{d}}\int_{Q(z_1,\rho)}|u-(u)_{z_1,\rho}|^{2-\frac{1}{d}}\, dz +N \rho^{d+2+\frac{1}{d}-\delta_1}.
\end{align*}
Applying Lemma \ref{lem1} and choosing $\delta_1=\frac{1}{2d}$, we obtain
$$\int_{Q(z_1,r)}|u-(u)_{z_1,r}|^{2-\frac{1}{d}}\, dz\leq N r^{d+2+\frac{1}{2d}}, $$
for any $r\in (0,\rho_0/4)$ and $z_1\in Q(z_0,\rho_0/4)$.
We then conclude that $u$ is H\"{o}lder continuous near $z_0$ by Campanato's characterization of H\"{o}lder continuity.

\section{H\"{o}lder Continuity Boundary Estimate}
\label{sec_holder_2}
In this section, we consider the case when $\W = \bR^d_+$. We again use a three-step approach to prove an $\epsilon$-regularity criterion near boundary. The main difference from the interior estimate is the iteration dealing with the pressure term.
\begin{theorem}
	\label{thm_holder_half}
	Let $(u,p)$  be a suitable weak solution of (\ref{NS1})-(\ref{NS2}) in $(0,T)\times\mathbb{R}_+^{d}$ satisfying (\ref{eqn_Ldcondition}). There exists a universal constant $\hat{\epsilon}_0$ satisfying the following property. Assume that for a point $\hat{z}=(\hat{t},\hat{x})$, where $\hat{x}=(x',0)$, and for some $\rho_0>0$ we have $Q^+(\hat z,\rho_0)\subset (0,T)\times \bR^d_+$ and
	$$A^+(\rho_0,\hat{z})+E^+(\rho_0,\hat{z})+D^+(\rho_0,\hat{z})\leq \hat{\epsilon}_0.$$
	Then $u$ is H\"{o}lder continuous near $\hat{z}$.
\end{theorem}
\subsection{Step 1} We present several  inequalities for the scale invariant quantities.
\begin{lemma}\label{lem_halfspace2}
	Suppose function $u\in W_2^1(B_r^+)$ with $r>0$ vanish on the boundary $x_d=0$. For any $q\in [2,2d/(d-2)]$, we have
	\begin{equation*}
			\int_{B^+_r}|u|^q \,dx \leq N(q,d)\left(\int_{B^+_r}\abs{\nabla u}^2\,dx\right)^{d(q/4-1/2)}\left(\int_{B^+_r}\abs{u}^2\,dx\right)^{q/2-d(q/4-1/2)}.
	\end{equation*}
\end{lemma}
\begin{proof}
	Modify the proof of Lemma \ref{lem2} using the Poincar\'{e} inequality with odd extension for functions vanishing on the flat boundary
	\begin{equation*}
		\int_{B^+_r}\abs{u}^2\, dx\leq Nr^2\int_{B^+_r}\abs{\nabla u}^2\, dx
	\end{equation*}
    to absorb the second term on the right-hand side.
\end{proof}

We recall the following two important lemmas which are useful in handling the estimates for the pressure  $p$.
\begin{lemma}\label{lem_halfspace4}
	Let $\mathcal{D}\subset \mathbb{R}^d$ be a  domain with smooth boundary and $T>0$ be a constant. Let $1<m<+\infty, 1<n<+\infty$ be two fixed integers. Assume that $g\in L_n^tL_m^x(\mathcal{D}_T).$ Then there exists a unique function pair $(v,p)$, which satisfies the following equation:
	\[\begin{cases}
	\partial_t v-\Delta v+\nabla p=g \quad &\text{in } \mathcal{D}_T,\\
	\nabla\cdot v=0 &\text{in } \mathcal{D}_T,\\
	[p]_{\mathcal{D}}(t)=0 & \text{for a.e. }t\in [0,T],\\
	v=0 & \text{on }\partial_p \mathcal{D}_T.
	\end{cases}\]
	Moreover, $v$ and $p$ satisfy the following estimate:
	$$\Vert v\Vert _{W^{1,2}_{n,m}(\mathcal{D}_T)}+\Vert p\Vert _{W^{0,1}_{n,m}(\mathcal{D}_T)}\leq C \Vert g\Vert _{L_n^tL_m^x(\mathcal{D}_T)},$$
	where the constant $C$ only depends on $m,n,T,$ and $\mathcal{D}$.
\end{lemma}

\begin{lemma}\label{lem_halfspace5}
	Let $1<m\leq 2, 1<n\leq 2$, and $m\leq s<+\infty$ be constants and $g\in L^t_nL^x_m(Q^+)$. Assume that the functions $v\in W^{0,1}_{n,m}(Q^+)$ and $p\in L_n^tL_m^x(Q^+)$ satisfy the equations:
	\[\begin{cases}
	\partial_tv-\Delta v+\nabla p=g \quad &\text{in } Q^+,\\
	\nabla\cdot v=0 &\text{in }Q^+,
	\end{cases}\]
	and the boundary condition
	$$ v=0, \quad \text{on }\{y\,\vert\, y=(y',0),|y'|<1\}\times[-1,0).$$
	Then, we have $v\in W^{1,2}_{n,s}(Q^+(1/2)),\, p\in W^{0,1}_{n,s}(Q^+(1/2))$, and
	\begin{align*}
		&\Vert v\Vert _{W^{1,2}_{n,s}(Q^+(1/2))}+\Vert p\Vert _{W^{0,1}_{n,s}(Q^+(1/2))}\\
		&\leq C \left(\Vert g\Vert _{L_n^tL_s^x(Q^+)}+\Vert v\Vert _{W^{0,1}_{n,m}(Q^+)}+\Vert p\Vert _{L_n^tL_m^x(Q^+)}\right),
	\end{align*}
	where the constant $C$ only depends on $m,n,$ and $s$.
\end{lemma}
 We refer the reader to \cite{Refer29} for the proof of Lemma \ref{lem_halfspace4}, and to \cite{Refer24b, Refer26b} for the proof of Lemma \ref{lem_halfspace5}.

 The following three lemmas are analogous to Lemma \ref{lem7}, \ref{lemAECD}, \ref{lem8}, and the proofs are similar.

\begin{lemma}\label{lem_halfspace7}
	For $\alpha\in[0,1]$, $p\in[2+\frac{4\alpha}{d},d+(4-d)\alpha]$, suppose $r>0$, $\hat{x}\in\partial\Omega$, $\omega(\hat{z},r)=Q^+(\hat{z},r)$, and $u$ satisfies the condition (\ref{eqn_Ldcondition}). Then we have
	\begin{align*}
		\frac{1}{r^{d+2-p}}\int_{Q^+(\hat{z},r)}|u|^p\, dz\leq N\left(A^+(r,\hat{z})+E^+(r,\hat{z})\right)^{\frac{d-p+2\alpha}{d-2}}.
	\end{align*}
In particular, taking $\alpha=1$ and $p = 4-2/d$  we have
$$C^+(r,\hat{z}) \leq N\left(A^+(r,\hat{z})+E^+(r,\hat{z})\right)^{\frac{d-2+2/d}{d-2}}.$$	
\end{lemma}

\begin{lemma}
	\label{lemAECD_half}
	Let $(u,p)$ be a pair of suitable weak solution of (\ref{NS1}). For  $\hat{x}\in \partial \Omega$ and  $\omega(\hat{z},\rho)=Q^+(\hat{z},\rho)$, we have
	\begin{equation}
		\label{eqn_halfspace25}
		A^+(\rho/2)+E^+(\rho/2)\leq N\left(C^+(\rho)^{\frac{d}{2d-1}}+C^+(\rho)^{\frac{3d}{2(2d-1)}}+ C^+(\rho)^{\frac{d}{2(2d-1)}}D^+(\rho)^{\frac{d}{2d-1}}\right),
	\end{equation}
	where $N$ is independent of $\hat{z}$ and $\rho$.
\end{lemma}

\begin{lemma}
	\label{lem2_half}Let $(u,p)$ be a pair of suitable weak solution of (\ref{NS1}) with $d\geq 4$. For constant $\gamma \in (0,1/2]$, $\rho >0$  and $\hat{x}\in\partial\Omega$, $\omega(\hat{z},\rho)=Q^+(\hat{z},\rho)$, we have
	\begin{align}\label{eqn_halfspace2}
			&\quad A^+(\gamma\rho)+E^+(\gamma\rho)\\
			&\leq N\left[\gamma^{2}A^+(\rho)+\gamma^{-d+1}\left((A^+(\rho)+E^+(\rho))^{\frac{d-1}{d-2}}+(A^+(\rho)+E^+(\rho))^{\frac{1}{2}}D^+(\rho)^{\frac{d}{2d-1}}\right)\right].\nonumber
	\end{align}
\end{lemma}

At last, we present an estimate for quantity $D^+(\rho)$, which is essentially different from Lemma \ref{lem_CD} for the interior case.

\begin{lemma}\label{lem_halfspace6}
	Let $(u,p)$ be a pair of suitable weak solution of (\ref{NS1}) and $u$ satisfies condition (\ref{eqn_Ldcondition}). Let $\gamma\in(0,1/4]$ and $\rho>0$ be constants. Suppose that $\hat{x}\in \partial \Omega$ and $\omega(\hat{z},\rho)=Q^+(\hat{z},\rho)$. Then given any small $\delta_2>0$,  we have
	\begin{align}\label{eqn_halfspace20}
		\begin{split}
			\quad D^+(\gamma\rho)\leq &N\left[ \gamma^{-d-2/d+2}(A^+(\rho)+E^+(\rho))^{1+\frac{2}{d(d-2)}}\right.\\
			&+\left.\gamma^{4-3/d-\delta_2}\left(D^+(\rho)+A^+(\rho)^{1-\frac{1}{2d}}+E^+(\rho)^{1-\frac{1}{2d}}\right)\right],
		\end{split}
	\end{align}
	where $N$ is a constant independent of $\gamma,\rho,$ and $\hat{z}$, but may depend on $\delta_2$.
\end{lemma}

\begin{proof}
	Without loss of generality, by shifting the coordinate we may assume that $\hat{z}=(0,0).$ By the scale-invariant property, we may also assume $\rho=1$. We  fix a domain $\tilde{B}\subset \mathbb{R}^d$ with smooth boundary so that
	$$B^+(1/2)\subset \tilde{B}\subset B^+,$$
	and denote $\tilde{Q}=\tilde{B}\times(-1,0)$. Using H\"{o}lder's inequality, Lemma \ref{lem_halfspace2} with $q=\frac{2d(d+2)}{d^2+4}$, and (\ref{eqn_Ldcondition}), we get
	\begin{align}\label{eqn_halfspace21}
		\begin{split}
			&\left(\int_{B^+}|u\cdot\nabla u|^{\frac{d(2d-1)}{d^2+2d-1}}\, dx\right)^{\frac{d^2+2d-1}{d^2}}\\
			&\leq \left(\int_{B^+}\abs{\nabla u}^2\, dx\right)^{1-\frac{1}{2d}}\left(\int_{B^+}\abs{u}^{\frac{2d(d+2)}{d^2+4}}\, dx\right)^{\frac{d^2+4}{2d^2(d-2)}}\left(\int_{B^+}\abs{u}^d\, dx \right)^{\frac{2(d-3)}{d(d-2)}}\\
				&\leq \left(\int_{B^+}\abs{\nabla u}^2\, dx\right)\left(\int_{B^+}\abs{u}^2\, dx\right)^{\frac{2}{d(d-2) }}\left(\int_{B^+}\abs{u}^d\, dx \right)^{\frac{2(d-3)}{d(d-2)}}\\
			&\leq N\left(\int_{B^+}|\nabla u|^2 \, dx\right)\left(\int_{B^+}|u|^2\, dx\right)^{\frac{2}{d(d-2)}}.
		\end{split}
	\end{align}
Integrating in $t$, we have $u\cdot \nabla u\in  W^{1,2}_{\frac{2d-1}{d},\frac{d(2d-1)}{d^2+2d-1}}(B^+)$.
	By Lemma \ref{lem_halfspace4}, there is a unique solution
	$$v\in W^{1,2}_{\frac{2d-1}{d},\frac{d(2d-1)}{d^2+2d-1}}(\tilde{Q}) \quad \text{and}\quad p_1\in W^{0,1}_{\frac{2d-1}{d},\frac{d(2d-1)}{d^2+2d-1}}(\tilde{Q})$$
	to the following initial boundary value problem:
	\[\begin{cases}
	\partial_t v-\Delta v+\nabla p_1=-u\cdot \nabla u \quad &\text{in } \tilde{Q},\\
	\nabla\cdot v=0 &\text{in } \tilde{Q},\\
	[p_1]_{\tilde{B}}(t)=0 & \text{for a.e. }t\in [0,T],\\
	v=0 & \text{on }\partial_p\tilde{Q}.
	\end{cases}\]
	Moreover, we have
	\begin{align}\label{eqn_halfspace22}
		\begin{split}
			&\Vert v \Vert_{L^t_{\frac{2d-1}{d}}L_{\frac{d(2d-1)}{d^2+2d-1}}^x (\tilde{Q})}+\Vert\nabla v \Vert_{L^t_{\frac{2d-1}{d}}L_{\frac{d(2d-1)}{d^2+2d-1}}^x (\tilde{Q})}\\
			&\quad +\Vert p_1 \Vert_{L^t_{\frac{2d-1}{d}}L_{\frac{d(2d-1)}{d^2+2d-1}}^x (\tilde{Q})}+\Vert \nabla p_1 \Vert_{L^t_{\frac{2d-1}{d}}L_{\frac{d(2d-1)}{d^2+2d-1}}^x (\tilde{Q})}\\
&\le N\|u\cdot\nabla u\|_{L^t_{\frac{2d-1}{d}}L_{\frac{d(2d-1)}{d^2+2d-1}}^x (\tilde{Q})}\\
			&\leq N\left(\int_{-1}^0\left(\int_{B^+}|\nabla u|^2 \, dx\right)\left(\int_{B^+}|u|^2\, dx\right)^{\frac{2}{d(d-2)}}\,dt\right)^{d/(2d-1 )},
		\end{split}
	\end{align}
	where in the last inequality we used (\ref{eqn_halfspace21}).
	
	We set $w=u-v$ and $p_2=p-p_1-[p]_{0,1/2}$. Then $w$ and $p_2$ satisfy
	\[\begin{cases}
	\partial_t w-\Delta w+\nabla p_2=0 \quad &\text{in } \tilde{Q},\\
	\nabla\cdot w=0 & \text{in } \tilde{Q},\\
	w=0 & \text{on } [-1,0) \times \{\partial\tilde{B}\cap\partial\Omega\}.\end{cases}\]
	By Lemma \ref{lem_halfspace5} and the triangle inequality, fixing some $s>0$ large enough to be specified later, we have $p_2\in W^{0,1}_{\frac{2d-1}{d},s}(Q^+(1/4))$ and
	\begin{align*}
		&\Vert \nabla p_2\Vert_{L^t_{\frac{2d-1}{d}}L_{s}^x(Q^+(1/4))}\\
		&\leq N\left[\Vert w\Vert_{L^t_{\frac{2d-1}{d}}L_{\frac{d(2d-1)}{d^2+2d-1}}^x(Q^+(1/2))}+\Vert \nabla w\Vert_{L^t_{\frac{2d-1}{d}}L_{\frac{d(2d-1)}{d^2+2d-1}}^x(Q^+(1/2))}\right.\\
		&\quad \left.+\Vert p_2\Vert_{L^t_{\frac{2d-1}{d}}L_{\frac{d(2d-1)}{d^2+2d-1}}^x(Q^+(1/2))}\right]\\
		&\leq N\left[ \Vert u\Vert_{L^t_{\frac{2d-1}{d}}L_{\frac{d(2d-1)}{d^2+2d-1}}^x(Q^+(1/2))}+\Vert \nabla u\Vert_{L^t_{\frac{2d-1}{d}}L_{\frac{d(2d-1)}{d^2+2d-1}}^x(Q^+(1/2))}\right.\\& \quad +\Vert p-[p]_{0,1/2}\Vert_{L^t_{\frac{2d-1}{d}}L_{\frac{d(2d-1)}{d^2+2d-1}}^x(Q^+(1/2))}+\Vert v\Vert_{L^t_{\frac{2d-1}{d}}L_{\frac{d(2d-1)}{d^2+2d-1}}^x(Q^+(1/2))}\\
		& \quad +\left. \Vert \nabla v\Vert_{L^t_{\frac{2d-1}{d}}L_{\frac{d(2d-1)}{d^2+2d-1}}^x(Q^+(1/2))}+\Vert p_1\Vert_{L^t_{\frac{2d-1}{d}}L_{\frac{d(2d-1)}{d^2+2d-1}}^x(Q^+(1/2))}\right].
	\end{align*}
	Together with (\ref{eqn_halfspace22}), we obtain
	\begin{align}\label{eqn_halfspace23}
		\begin{split}
			&\Vert \nabla p_2\Vert_{L^t_{\frac{2d-1}{d}}L_{s}^x(Q^+(1/4))}\\
			&\leq N\left[ \Vert u\Vert_{L^t_{\frac{2d-1}{d}}L_{\frac{d(2d-1)}{d^2+2d-1}}^x(Q^+(1/2))}+\Vert \nabla u\Vert_{L^t_{\frac{2d-1}{d}}L_{\frac{d(2d-1)}{d^2+2d-1}}^x(Q^+(1/2))}\right.\\
			&\left.\quad+\Vert p-[p]_{0,1/2}\Vert_{L^t_{\frac{2d-1}{d}}L_{\frac{d(2d-1)}{d^2+2d-1}}^x(Q^+(1/2))}\right. \\
			&\left.\quad +\left(\int_{-1}^0\left(\int_{B^+}|\nabla u|^2 \, dx\right)\left(\int_{B^+}|u|^2\, dx\right)^{\frac{2}{d(d-2)}}\,dt\right)^{d/(2d-1)}\right].
		\end{split}
	\end{align}
	Recall that $0<\gamma\leq 1/4$. Then by using the Sobolev-Poincar\'{e} inequality, the triangle inequality, (\ref{eqn_halfspace22}), (\ref{eqn_halfspace23}), and H\"{o}lder's inequality, we bound $D^+(\gamma)$ by
	\begin{align*}		
&\frac{N}{\gamma^{d+2/d-2}}\int_{-\gamma^2}^0\left(\int_{B^+(\gamma)}|\nabla p|^{\frac{d(2d-1)}{d^2+2d-1}}\,dx\right)^{\frac{d^2+2d-1}{d^2}}dt\\
&\le \frac{N}{\gamma^{d+2/d-2}}\int_{-\gamma^2}^0\left(\left(\int_{B^+(\gamma)}|\nabla p_1|^{\frac{d(2d-1)}{d^2+2d-1}}\,dx\right)^{\frac{d^2+2d-1}{d^2}}+\left(\int_{B^+(\gamma)}|\nabla p_2|^{\frac{d(2d-1)}{d^2+2d-1}}\,dx\right)^{\frac{d^2+2d-1}{d^2}}\right)dt\\
		& \leq N\gamma^{-d-2/d+2}E^+(1)A^+(1)^{\frac{2}{d(d-2)}}+N\gamma^{4-3/d-(2d-1)/s}\int_{-\gamma^2}^0\left(\int_{B^+(\gamma)}|\nabla p_2|^{s}\,dx\right)^{\frac{2d-1}{sd}} dt\\
		&\leq N\left[\gamma^{-d-2/d+2}E^+(1)A^+(1)^{\frac{2}{d(d-2)}}\right.\\
		&\quad\left.+\gamma^{4-3/d-(2d-1)/s}\left(D^+(1) +A^+(1)^{1-\frac{1}{2d}}+E^+(1)^{1-\frac{1}{2d}}\right)\right].
	\end{align*}
	By making $s$ large such that $\frac{2d-1}{s}<\delta_2$, we finish the proof.
\end{proof}
\begin{corollary}
	\label{col6_half}
Let $(u,p)$ be a pair of suitable weak solution of (\ref{NS1}) with $d\geq 4$. Suppose there exist  $\hat{x}\in \partial \Omega$,  $C_1>0$, and $\rho_0>0$ such that $\omega(\hat{z},\rho)=Q^+(\hat{z},\rho)$, $C^+(\rho,\hat{z})\leq C_1$ for all $\rho\in(0,\rho_0]$ and $D^+(\rho_{0},\hat{z})\leq C_1$. Then we can find $C:=C(C_1)>0$ such that $D^+(\rho,z^{\ast})\leq C$ for all $z^{\ast}\in Q(\hat{z},\rho_0/2)\cap\{x_d=0\}$ and $\rho\in(0,\rho_0/4)$.
\end{corollary}
\begin{proof}
	Because $Q^+(z^{\ast},\rho_0/2)\subset Q^+(\hat{z},\rho_0)$ for $z^{\ast}\in Q(\hat{z},\rho_0/2)\cap\{x_d=0\}$, we have $D^+(\rho_0/2,z^{\ast})\leq N D^+(\rho_0,\hat{z})$.
	From (\ref{eqn_halfspace25}) and (\ref{eqn_halfspace20}) with $\delta_2=1$, we can get an estimation for $D^+$.
	\begin{align}
		\begin{split}\label{eqn80}
			D^+\left(\frac{\gamma\rho}{4},z^{\ast}\right)\leq& N\left[\gamma^{-d-2/d+2}\left(A^+\left(\frac{\rho}{4},z^{\ast}\right)+E^+\left(\frac{\rho}{4},z^{\ast}\right)\right)^{1+\frac{2}{d(d-2)}}\right.\\
			&\left. +\gamma^{3-3/d}\left(D^+\left(\frac{\rho}{4},z^{\ast}\right)+A^+\left(\frac{\rho}{4},z^{\ast}\right)^{1-\frac{1}{2d}}+E^+\left(\frac{\rho}{4},z^{\ast}\right)^{1-\frac{1}{2d}}\right)\right]\\
			\leq&N(C_1)\left[\gamma^{-d-2/d+2}\left(1+D^+\left(\frac{\rho}{2},z^{\ast}\right)^{\frac{d}{2d-1}}\right)^{1+\frac{2}{d(d-2)}}\right.\\
			&\left.+\gamma^{3-3/d}\left(D^+\left(\frac{\rho}{4},z^{\ast}\right)+\left(1+D^+\left(\frac{\rho}{2},z^{\ast}\right)^{\frac{d}{2d-1}}\right)^{1-\frac{1}{2d}}\right)\right]\\
			\leq&N(C_1)\left[\gamma^{-d-2/d+2}\left(1+D^+\left(\frac{\rho}{2},z^{\ast}\right)^{\frac{d^2-2d+2}{2d^2-5d+2}}\right)\right.\\
			&+\left. \gamma^{3-3/d}\left(D^+\left(\frac{\rho}{2},z^{\ast}\right)+D^+\left(\frac{\rho}{2},z^{\ast}\right)^{1/2}+1\right)\right].
		\end{split}
	\end{align}
By Young's inequality, we have
$$D^+(\rho/2,z^{\ast})^{1/2}\leq \frac{1}{2}D^+(\rho/2,z^{\ast})+\frac{1}{2}.$$
For any  $\epsilon>0$ and  $\delta: = 1-\frac{d^2-2d+2}{2d^2-5d+2}$, which is positive when $d\geq 4$. Again by Young's inequality we get
\begin{align*}
	D^+(\rho/2,z^{\ast})^{1-\delta}&=(\epsilon D^+(\rho/2,z^{\ast}))^{1-\delta}\frac{1}{\epsilon^{1-\delta}}  \leq (1-\delta)\epsilon D^+(\rho/2,z^{\ast})+\delta\epsilon^{-\frac{1-\delta}{\delta}}.
\end{align*}
We can choose $\gamma$ and $\epsilon$ small such that $N\gamma^{3-3/d}<1/8$ and $N\gamma^{-d-2/d+2}(1-\delta)\epsilon<1/8$. The two inequalities above implies that (\ref{eqn80}) can be written into such form:
\begin{equation*}
D^+\left(\gamma\rho/4,z^{\ast}\right)\leq \frac{1}{2} D^+(\rho/2,z^{\ast})+C.
\end{equation*}
 The rest of the proof is a handy modification of Corollary \ref{col_CD}.
\end{proof}

\subsection{Step 2}
We will find some decay rates for the scale invariant quantities with respect to the radius of the cylinder assuming the quantities are initially small.
\begin{lemma}
	\label{lem_initial_halfspace}
	There exists a universal constant $\hat{\epsilon}_0>0$ satisfying the following property. Suppose that for some $\hat{z}=(\hat{x},\hat{t})$, where $\hat{x}=(x',0)$, and for some $\rho_0>0$,  it holds that $\omega(\hat{z},\rho_0)=Q^+(\hat{z},\rho_0)$ and
	\begin{equation}\label{eqn_halfspace5}
	A^+(\rho_0,\hat{z})+E^+(\rho_0,\hat{z})+D^+(\rho_0,\hat{z})^{1/\tau}\leq \hat{\epsilon}_0,
	\end{equation}
	where $\tau=1-\frac{1}{2d}+\epsilon$, $\epsilon\sim\mathcal{O}\left(\frac{1}{d^2}\right)$.
	Then fixing any $\alpha_0\in(0,2)$,  we can find $N>0$ such that for any $\rho
	\in(0,\rho_0/4)$ and $z^{\ast}\in \overline{Q^+(\hat{z},\rho_0/4)}$, the following estimate holds uniformly
	\begin{equation}\label{eqn_halfspace6}
	A^+(\rho,z^{\ast})+E^+(\rho,z^{\ast})
+C^+(\rho,z^{\ast})^{\frac{d-2}{d-2+2/d}}+D^+(\rho,z^{\ast})\leq N\hat{\epsilon}_0^{\alpha_0^2\tau/4}\left(\frac{\rho}{\rho_0}\right)^{\alpha_0^2\tau/2},
	\end{equation}
 where $N$ is a positive constant depending on $\alpha_0$, but independent of $\hat{\epsilon}_0$, $\rho_0$, $\rho$, and $z^{\ast}$.
\end{lemma}

\begin{proof}
We divide the proof into two parts. In the first part, we only consider $z^*$  on the boundary $ Q(\hat{z},\rho_0/2)\cap \{x_d=0\}$. In  the second part, we use an iteration argument to close the proof for general $z^*\in Q^+(\hat{z},\rho_0/4)$.

i) First we assume $z^{\ast}\in Q(\hat{z},\rho_0/2)\cap \{x_d=0\}$. In this case we will prove a slightly stronger estimate than (\ref{eqn_halfspace6}):
\begin{equation}\label{eqn_halfspace66}
A^+(\rho,z^{\ast})+E^+(\rho,z^{\ast})+C^+(\rho,z^{\ast})^{\frac{d-2}{d-2+2/d}}+D^+(\rho,z^{\ast})^{1/\tau}\leq N\hat{\epsilon}_0^{\alpha_0/2}\left(\frac{\rho}{\rho_0}\right)^{\alpha_0}.
\end{equation}
 By (\ref{eqn_halfspace5}) and
$$Q^+(z^{\ast},\rho_0/2) \subset Q^+(\hat{z},\rho_0),$$
we get
\begin{equation}
\label{eqn_halfspace_step2}
A^+(\rho_1,z^{\ast})+E^+(\rho_1,z^{\ast})+D^+(\rho_1,z^{\ast})^{1/\tau}\leq N\hat{\epsilon}_0,
\end{equation}
where $\rho_1=\rho_0/2.$
By Lemma \ref{lem_halfspace7},
\begin{equation}
                \label{eq8.04}
C^+(\rho_1,z^*)\leq (N\hat{\epsilon}_0)^{\frac{d-2+2/d}{d-2}}.
\end{equation}
Next we fix an auxiliary parameter  $\alpha\in(\alpha_0,2)$. By a scaling argument, we first discuss a special case when $\rho_1^{\alpha} = N\hat{\epsilon}_0<1$. In this case, we can prove the following decay rates inductively:
\begin{equation}\label{eqn_halfspace13}
A^+(\rho_k)+E^+(\rho_k)\leq \rho_k^{\alpha},\quad C^+(\rho_k)^{\frac{d-2}{d-2+2/d}}\leq \rho_k^{\alpha}, \quad D^+(\rho_k)\leq \rho_k^{\alpha\tau},
\end{equation}
where $\rho_{k+1}=\rho_k^{1+\beta}=\rho_1^{(1+\beta)^k}$,  and $\beta$ is a small number to be specified. For $k=1$, the statement follows from  (\ref{eqn_halfspace_step2}) and \eqref{eq8.04}. Next  by choosing $\gamma=\rho_k^{\beta}$ and $\rho=\rho_k$ in (\ref{eqn_halfspace2}), we have
$$
		A^+(\rho_{k+1})+E^+(\rho_{k+1})\leq N\left[ \rho_k^{2\beta+\alpha}+\rho_k^{(-d+1)\beta}\left(\rho_k^{\frac{d-1}{d-2}\alpha}+\rho_k^{\left(\frac{1}{2}+\frac{d}{2d-1}\tau\right)\alpha}\right)\right].
$$
We choose $\beta$ satisfying
$$\beta < \min\left\{\frac{\alpha}{(d-2)(d+\alpha-1)},\frac{\left(\frac{d\tau}{2d-1}-\frac{1}{2}\right)\alpha}{d+\alpha-1} \right\}\sim \mathcal{O}\left(\frac{1}{d^2}\right). $$
Then all the exponents on the right-hand sides are greater than $(1+\beta)\alpha$.

To estimate the remaining term $D^+(\rho_{k+1})$, we apply Lemma \ref{lem_halfspace6} but with different step size. Let $\beta_1=(1+\beta)^{n_0+1}-1$, where $n_0$ is an integer to be specified later. Instead of plugging in the result of one last previous step, we plug in (\ref{eqn_halfspace20}) with $\gamma = \rho_{k-n_0}^{\beta_1}$ and $\rho=\rho_{k-n_0}$, and we have

\begin{align}\label{eqn_halfspace26}
	D^+(\rho_{k+1})\leq & N\left[\rho_{k-n_0}^{(-d-2/d+2)\beta_1+\left(1+\frac{2}{d(d-2)}\right)\alpha}+ \rho_{k-n_0}^{(4-3/d-\delta_2)\beta_1}\left(\rho_{k-n_0}^{\alpha\tau}+\rho_{k-n_0}^{\alpha(1-\frac{1}{2d})}\right)\right]
\end{align}

Our goal is to choose an appropriate $\beta_1$ such that all three exponents on right-hand side of (\ref{eqn_halfspace26}) are greater than $(1+\beta_1)\alpha\tau$. We hence obtain an upper bound and a lower bound for $\beta_1$:
\begin{equation*}
	\beta_1 > \frac{\left(\tau-\left(1-\frac{1}{2d}\right)\right)\alpha}{4-\frac{3}{d}-\alpha\tau} = \frac{\epsilon\alpha}{4-\frac{3}{d}-\alpha\tau},
\end{equation*}
\begin{equation*}
\beta_1 < \frac{\left(1+\frac{2}{d(d-2)}-\tau\right)\alpha}{d+\frac{2}{d}-2+\alpha\tau} = \frac{\left(\frac{1}{2d}+\frac{2}{d(d-2)}-\epsilon\right)\alpha}{d+\frac{2}{d}-2+\alpha\tau}\sim \mathcal{O}\left(\frac{1}{d^2}\right).
\end{equation*}
To ensure such $\beta_1$ exists, we make $\epsilon\sim\mathcal{O}\left(\frac{1}{d^2}\right)$ small such that the upper bound is greater than the lower bound of $\beta_1$. As long as $\beta$ is small enough, there exists an integer $n_0$ such that $\beta_1=(1+\beta)^{n_0+1}-1$ satisfies the conditions above.

Now we can find $\xi>0$ depending on $\beta$ and that
\begin{equation*}
A^+(\rho_{k+1})+E^+(\rho_{k+1})\leq N\rho_{k+1}^{\alpha+\xi}, \quad D^+(\rho_{k+1})\leq N\rho_{k+1}^{\alpha\tau+\xi},
\end{equation*}
where $N$ is a constant independent of $k$ and $\xi$ . By taking $\hat{\epsilon}_0$ small enough such that $ N\rho_{k+1}^{\xi}< N\rho_1^{\xi} <N(N\hat{\epsilon}_0)^{\xi/2}<1$ , we obtain
\begin{equation*}
A^+(\rho_{k+1})+E^+(\rho_{k+1})\leq \rho^{\alpha}_{k+1}, \quad D^+(\rho_{k+1})\leq \rho^{\alpha\tau}_{k+1}.
\end{equation*}
By induction, we have justified (\ref{eqn_halfspace13}) for the case when $\rho_1^{\alpha} = N\hat{\epsilon}_0$.

For convenience,   we additionally assume that the parameter $\beta$ satisfying
$$
\alpha_0 < \min\left\{\frac{1}{1+\beta}(\alpha-(d-2)\beta),\frac{1}{(1+\beta)\tau}\left(\alpha\tau-\left(d+\frac{2}{d}-2\right)\beta\right)\right\},$$
There always exists feasible $\beta$ because $\alpha>\alpha_0$.

Now for any $\rho\in(0,\rho_0/2)$, we can find a positive integer $k$ such that $\rho_{k+1}\leq \rho<\rho_k$. Then
$$A^+(\rho)+E^+(\rho)\leq \left(\frac{\rho_k}{\rho_{k+1}}\right)^{d-2}(A^+(\rho_k)+E^+(\rho_k))\leq \rho_k^{\alpha-(d-2)\beta}=\rho_{k+1}^{\frac{1}{1+\beta}(\alpha-(d-2)\beta)}\leq \rho_{k+1}^{\alpha_0},$$
$$D^+(\rho)\leq \left(\frac{\rho_k}{\rho_{k+1}}\right)^{d+\frac{2}{d}-2}D^+(\rho_k)\leq \rho_k^{\alpha\tau-(d+\frac{2}{d}-2)\beta}=  \rho_{k+1}^{\frac{1}{1+\beta}(\alpha\tau-(d+\frac{2}{d}-2)\beta)}\leq \rho_{k+1}^{\alpha_0\tau}.$$
Hence we have proved the statement of the lemma when $\rho_1^{\alpha} = N\hat{\epsilon}_0$. For general $\rho_0>0$, we  use the scale invariant property of the quantities and yield similar results with an additional scaling factor $N\hat{\epsilon}_0^{\alpha_0/\alpha}\rho_{0}^{-\alpha_0}$   on the right-hand side. Hence (\ref{eqn_halfspace66}) is true.

ii) To deal with $z^{\ast}\in Q^+(\hat{z},\rho_0/4)$, we need to discuss two cases as comparing $x^*_d$, the distance of $z^*$ to the boundary, with $\rho$, the radius of the cylinder.

 When $\rho\ge x^*_d$, we denote the projection of $z^*$ on the boundary by $\hat{z}^*$. Because $\w(z^*,\rho)\subset Q^+(\hat{z}^*,2\rho)$, by definition we have $$A^+(\rho,z^*)+E^+(\rho,z^*)+C^+(\rho,z^*)^{\frac{d-2}{d-2+2/d}}\leq N\left(A^+(2\rho,\hat{z}^*)+E^+(2\rho,\hat{z}^*)+C^+(2\rho,\hat{z}^*)^{\frac{d-2}{d-2+2/d}}\right).$$ By the triangle inequality, we have
\begin{equation*}
\int_{\omega(z^*,\rho)}\abs{p-[p]_{x^*,\rho}(t)}^{2-1/d}\, dz\leq N\int_{Q^+(\hat{z}^{\ast},2\rho)}\abs{p-[p]_{\hat{x}^{\ast},2\rho }(t)}^{2-1/d}\,dz.
\end{equation*}
Hence $D^+(\rho,z^*)\leq ND^+(2\rho,\hat{z}^*)$.
From part (i) we know
	$$A^+(2\rho,\hat{z}^{\ast})+E^+(2\rho,\hat{z}^{\ast})+C^+(2\rho,\hat{z}^{\ast})^{\frac{d-2}{d-2+2/d}}+D^+(2\rho,\hat{z}^{\ast})^{1/\tau}\leq N\hat{\epsilon}_0^{\alpha_0/2}\left(\frac{\rho}{\rho_0}\right)^{\alpha_0}.$$
	The three inequalities above together imply that
	$$A^+(\rho,z^{\ast})+E^+(\rho,z^{\ast})+C^+(\rho,z^{\ast})^{\frac{d-2}{d-2+2/d}}+D^+(\rho,z^{\ast})^{1/\tau}\leq N\hat{\epsilon}_0^{\alpha_0/2}\left(\frac{\rho}{\rho_0}\right)^{\alpha_0}.$$
	
	When $\rho<x^*_d$, we have $\w(z^*,\rho) = Q(z^*,\rho) \subset Q(z^*,x^*_d)\subset Q^+(\hat{z}^*,2x^*_d)$. By the proof above, we have $$A^+(x^*_d,z^{\ast})+E^+(x^*_d,z^{\ast})+C^+(x^*_d,z^{\ast})^{\frac{d-2}{d-2+2/d}}+D^+(x^*_d,z^{\ast})^{1/\tau}\leq N\hat{\epsilon}_0^{\alpha_0/2}\left(\frac{x^*_d}{\rho_0}\right)^{\alpha_0}.$$
	With $\hat{\epsilon}_0$ small such that $ N\hat{\epsilon}_0^{\alpha_0/2}\left(\frac{x^*_d}{\rho_0}\right)^{\alpha_0}<\epsilon_0^{1/\tau} $ where $\epsilon_0$ is from Lemma \ref{lem_a0_whole}, we can apply the interior result of Lemma \ref{lem_a0_whole} to obtain
$$
A^+(\rho,z^{\ast})+E^+(\rho,z^{\ast})
+C^+(\rho,z^{\ast})^{\frac{d-2}{d-2+2/d}}+D^+(\rho,z^{\ast})\leq N\hat{\epsilon}_0^{\alpha_0^2\tau/4}\left(\frac{\rho}{\rho_0}\right)^{\alpha_0^2\tau/2}.$$
	The proof is complete.
\end{proof}

\subsection{Step 3} In this final step, we first use parabolic $L_p$ estimates to further improve the decay rate and then conclude the result by using Campanoto's characterization of H\"{o}lder continuity.
By (\ref{eqn_halfspace6}) from the previous step, we know the following estimates are true for all $\rho>0$ sufficiently small and $z^{\ast}=(t^{\ast},x^{\ast})\in Q(\hat{z},\rho_0/4)\cap \{x_d=0\}$:
\begin{equation}\label{eqn_halfspace15}
\int_{Q^+(z^{\ast},\rho)}|u|^{4-\frac{2}{d}}\, dz\leq N\rho^{d+\frac{2}{d-2}-\frac 1 d+\epsilon},
\end{equation}
\begin{equation}\label{eqn_halfspace16}
\int_{Q^+(z^{\ast},\rho)}|p-[p]_{x^{\ast},\rho}|^{2-\frac{1}{d}}\, dz\leq N\rho^{d+\frac{1}{d}+\epsilon},
\end{equation}
where $\epsilon\sim \mathcal{O}\left(\frac{1}{d^2}\right)$.

Let $v$ be the unique weak solution to the heat equation
$$\partial_t v-\Delta v=0 \quad \text{in }  Q^+(z^{\ast},\rho)$$
with the boundary condition $v=u$ on $\partial_pQ^+(z^{\ast},\rho)$.

Let $0<r<\rho$. By the Poincar\'{e} inequality with zero boundary condition and using the fact that $L_{\infty}$ norm of the gradient of a caloric function in a small half cylinder is controlled by any $L_p$ norm of it in a larger half cylinder, we have
\begin{align}\label{eqn_halfspace18}
	\begin{split}
		\int_{Q^+(z^{\ast},r)}\abs{v}^{2-\frac{1}{d}}\, dz
		&\leq N r^{d+4-\frac{1}{d}}\norm{\nabla v}_{L_{\infty}(Q^+(z^*,r))}\\
		& \leq N\left(\frac{r}{\rho}\right)^{d+4-\frac{1}{d}}\int_{Q^+(z^{\ast},\rho)}|v|^{2-\frac{1}{d}}\, dz.
	\end{split}
\end{align}

Denote $w=u-v$. Then $w$ satisfies the inhomogeneous heat equation
$$\partial_t w_i-\Delta w_i=-\partial_j(u_iu_j)-\partial_i(p-[p]_{x^{\ast},\rho}) \quad \text{in } Q^+(z^{\ast},\rho)$$
with the zero boundary condition. By the classical $L_p$ estimate for the heat equation, we have
$$\Vert\nabla w\Vert_{L_{2-\frac{1}{d}}(Q^+(z^{\ast},\rho))}\leq N\left[\left\lVert |u|^2 \right\rVert_{L_{2-\frac{1}{d}}(Q^+(z^{\ast},\rho))} + \left\lVert p-[p]_{x^{\ast},\rho}\right\rVert_{L_{2-\frac{1}{d}}(Q^+(z^{\ast},\rho))}  \right],$$
which together with (\ref{eqn_halfspace15}) and (\ref{eqn_halfspace16})  yields
\begin{equation*}
\int_{Q^+(z^{\ast},\rho)}|\nabla w|^{2-\frac{1}{d}}\, dz\leq N \rho^{d+\frac{1}{d}+\epsilon}.
\end{equation*}
By the Poincar\'{e} inequality with  zero boundary condition, we get
\begin{equation}\label{eqn_halfspace19}
\int_{Q^+(z^{\ast},\rho)}|w|^{2-\frac{1}{d}}\, dz\leq N \rho^{d+2+\epsilon}.
\end{equation}
Using (\ref{eqn_halfspace18}), (\ref{eqn_halfspace19}), and the triangle inequality, we have
\begin{align*}
	&\int_{Q^+(z^{\ast},r)}|u|^{2-\frac{1}{d}}\, dz \\
	&\leq \int_{Q^+(z^{\ast},r)}|v|^{2-\frac{1}{d}}\, dz+\int_{Q^+(z^{\ast},r)}|w|^{2-\frac{1}{d}}\, dz\\
	& \leq N\left(\frac{r}{\rho}\right)^{d+4-\frac{1}{d}}\int_{Q^+(z^{\ast},\rho)}|v|^{2-\frac{1}{d}}\, dz + N r^{d+2+\epsilon}\\
	&\leq N\left(\frac{r}{\rho}\right)^{d+4-\frac{1}{d}}\int_{Q^+(z^{\ast},\rho)}|u|^{2-\frac{1}{d}}\, dz   +N\left(\frac{r}{\rho}\right)^{d+4-\frac{1}{d}}\int_{Q^+(z^{\ast},\rho)}|w|^{2-\frac{1}{d}}\, dz + N r^{d+2+\epsilon}\\
	& \leq N\left(\frac{r}{\rho}\right)^{d+4-\frac{1}{d}}\int_{Q^+(z_0,\rho)}|u|^{2-\frac{1}{d}}\, dz +N \rho^{d+2+\epsilon}.
\end{align*}
Applying Lemma \ref{lem1},  we obtain
\begin{equation}
\label{eqn_decay}
\int_{Q^+(z^{\ast},r)}|u|^{2-\frac{1}{d}}\, dz\leq N r^{d+2+\epsilon}
\end{equation}
for any $r\leq \rho_0/4$ and $z^{\ast}\in Q(\hat{z},\rho_0/4)\cap\{x_d=0\}$.

 Consider any $\tilde{z}=(\tilde{t},\tilde{x})\in Q^+(\hat{z},\rho_0/8)$. Let  $z^{\ast}=(\tilde{t},\tilde{x}',0)$ be the projection of $\tilde{z}$ on the boundary. Note that $z^{\ast}\in Q(\hat{z},\rho_0/8)\cap\{x_d=0\}$.
We consider two cases either the radius of the parabolic ball around $z^{\ast}$ is smaller or larger than $\tilde{x}_d$.

	\textsl{Case 1:} $\tilde{x}_d\leq r$. In this case, we have $\omega(\tilde{z},r)\subset Q^+(z^{\ast},2r)$. Thus by (\ref{eqn_decay}), we have
	\begin{equation*}
	\int_{\omega(\tilde{z},r)}\abs{u-(u)_{\tilde{z},r}}^{2-1/d}\, dz\leq N\int_{Q^+(z^{\ast},2r)}\abs{u}^{2-1/d}\,dz\leq Nr^{d+2+\epsilon}.
	\end{equation*}
	
	\textsl{Case 2:} $r< \tilde{x}_d$. For $r<\rho\leq \tilde{x}_d$, from the proof of Theorem \ref{thm_holder_whole}, we have
	\begin{align}
		\label{eqn_32}
		\int_{Q(\tilde{z},r)}\abs{u-(u)_{\tilde{z},r}}^{2-1/d}\,dz
		\leq  N\left(\frac{r}{\rho}\right)^{d+4-\frac{1}{d}} \int_{Q(\tilde{z},\rho)}\abs{u-(u)_{\tilde{z},\rho}}^{2-1/d}\,dz + N\rho^{d+2+\epsilon}.
	\end{align}
    We apply  Lemma \ref{lem1} to (\ref{eqn_32}) to get
    \begin{align}
    	\label{eqn_33}
    	\int_{Q(\tilde{z},r)}\abs{u-(u)_{\tilde{z},r}}^{2-1/d}\,dz
    	\leq  N\left(\frac{r}{\tilde{x}_d}\right)^{d+2+\epsilon} \int_{Q(\tilde{z},\tilde{x}_d)}\abs{u-(u)_{\tilde{z},\tilde{x}_d}}^{2-1/d}\,dz + Nr^{d+2+\epsilon}.
    \end{align}
By Case 1, we have
$$ \int_{Q(\tilde{z},\tilde{x}_d)}\abs{u-(u)_{\tilde{z},\tilde{x}_d}}^{2-1/d}\,dz \leq N\tilde{x}_d^{d+2+\epsilon}.$$
Plug this into (\ref{eqn_33}) to get
	\begin{equation*}
 \int_{Q(\tilde{z},r)}\abs{u-(u)_{\tilde{z},r}}^{2-1/d}\,dz \leq Nr^{d+2+\epsilon}.
	\end{equation*}
By Campanato's characterization of H\"{o}lder continuity near a flat boundary (see, for instance, \cite[Lemma 4.11]{Lieberman1}), we can conclude that $u$ is H\"{o}lder continuous in a neighborhood of $\hat{z}$.

\section{Proof of Theorem \ref{thm_main}}
\label{sec_thm1}
In this section, we start with a construction on a sequence of suitable weak solutions which converges to a limiting solution. Let $(u,p)$ be a pair of Leray-Hopf weak solution of the Cauchy problem (\ref{NS1})-(\ref{NS2}) on $\mathbb{R}^d$ or $\mathbb{R}^d_+$. Because of the local strong solvability for smooth data and the weak-strong uniqueness (see, for instance, \cite{Wahl1}), we know that $u$ is regular for $t\in (0,T_0)$ for some $T_0\in(0,T]$. Suppose $T_0$ is the first blowup time of $u$, and $Z_0=(T_0,X_0)=(T_0,X_{0,1},X_{0,2},\ldots,X_{0,d})=(T_0,X'_0,X_{0,d})$ is a singular point. We take a decreasing sequence $\{\lambda_k\}$ converging to 0 and rescale the pair $(u,p)$ at time $T_0$. Define
$$u_k(t,x) = \lambda_k u(T_0+\lambda_k^2t,X_0+\lambda_kx),\quad p_k(t,x) = \lambda_k^2p(T_0+\lambda_k^2t,X_0+\lambda_kx),$$
for each $k=1,2,\ldots$. We will show that  each $(u_k,p_k)$ is a suitable weak solution of (\ref{NS1})-(\ref{NS2}) and $u_k$ is smooth for $t\in (-\lambda_k^{-2}T_0,0)$.

To prove this, the first observation is the property of uniform boundedness of the scale invariant quantities after the rescaling. In  this section, we use the ambiguous notation as we mentioned before in the preliminaries. By $\R^d_{(+)}$ we mean either $\R^d$ or $\R^d_+$ depending on which domain  we are talking about: the whole space or the half space.  The same goes for  $A^{(+)}, E^{(+)}, C^{(+)}, D^{(+)}$, etc.
\begin{lemma}
	\label{lem_bdd}
	Under the assumptions in Theorem \ref{thm_main}, there exists $N>0$  such that for any $z_0\in (-\infty,0]\times\R^d_{(+)}$ and $0<r\leq 1$,
	\begin{equation}
	\label{eqn_Cbdd}
	\limsup_{k\rightarrow \infty} C^{(+)}(r,z_0,u_k,p_k)\leq N,
	\end{equation}
	and
	\begin{equation}
	\label{eqn_Dbdd}
	\limsup_{k\rightarrow \infty}D^{(+)}(r,z_0,u_k,p_k)\leq N.
	\end{equation}
\end{lemma}

\begin{proof}
For  $z_0\in (-\infty,0]\times\R^d_{(+)}$, denote
$$z_0^k = (t_0^k,x_0^k)=(T_0+\lambda^2_k t_0,X_0+\lambda_k x_0).$$
	
For convenience, we first assume $\W = \R^d$ to prove (\ref{eqn_Cbdd}).
Since $C$ and $D$ are invariant, we have
$$C(r,z_0,u_k,p_k) = C(\lambda_k r,z_0^k, u,p),$$
and
$$D(r,z_0,u_k,p_k) = D(\lambda_k r,z_0^k, u,p).$$

Using H\"{o}lder's inequality we have
\begin{align}\label{eqn56}
	\begin{split} C(r,z_0,u,p)&=\frac{1}{r^{d+2/d-2}}\int_{Q(z_0,r)}\abs{u}^{\frac{2(2d-1)}{d}}\, dz\\
			&\leq  \frac{1}{r^{d+2/d-4}} \esssup_{t_0-r^2\le t\le t_0}\int_{B(x_0,r)}\abs{u}^{\frac{2(2d-1)}{d}}\, dx\\
			&\leq N\frac{1}{r^{d+2/d-4}}\esssup_{t_0-r^2\le t\le t_0}\left(\int_{B(x_0,r)}\abs{u}^d\, dx\right)^{\frac{2(2d-1)}{d^2}}\left(r^d\right)^{1-\frac{2(2d-1)}{d^2}}\\
			&\leq N\left(\esssup_{t_0-r^2\le t\le t_0}\int_{B(x_0,r)}\abs{u}^d\, dx\right)^{\frac{2(2d-1)}{d^2}}\\
			&\leq N\norm{u}_{L_{\infty}^tL_d^x(Q(z_0,r))}^{\frac{2(2d-1)}{d}}.
		\end{split}
	\end{align}
   Substituting $r$ with $\lambda_k r$ and $z_0$ with $z_0^k$, we have
   $$ C(\lambda_k r,z_0^k, u,p)\leq N\norm{u}_{L_{\infty}^tL_d^x(Q(z_0^k,\lambda_k r))}^{\frac{2(2d-1)}{d}}\leq N,$$
   where in the last inequality we  used (\ref{eqn_Ldcondition}). This part of proof can easily be adapted to the case $\W = \R^d_+$.

    To prove (\ref{eqn_Dbdd}), we need to consider several cases separately:

     \textsl{i)} $\W=\R^d$, by using the Calder\'{o}n-Zygmund estimate, one has
    \begin{equation*}
    	\norm{p}_{L_{\infty}^tL_{d/2}^x((0,T)\times\R^d)}\leq K.
    \end{equation*}
    Since $\frac{d}{2}\geq 2-\frac{1}{d}$, following the same reasoning in (\ref{eqn56}) we can reach (\ref{eqn_Dbdd}).

     \textsl{ii)} $\W=\R^d_+$ and $X_{0,d}>0$, that is the half-space case when $Z_0$ does not lie on the boundary. The domain of $u_k$ will expand to the whole space, therefore $z_0\in(-\infty,0]\times \bR^d$. When $X_{0,d}\geq 1/4$, i.e., $Z_0$ is away from the boundary,  from (\ref{presure_bd}) we know $D(1/4,Z_0,u,p)\leq N$. When $k$ is large, $Q(z_0^k,\lambda_k r)\subset Q(Z_0,1/4)$. By Corollary \ref{col_CD}, we know $D(r,z_0,u_k,p_k)=D(\lambda_kr,z_0^k,u,p)$ is uniformly bounded. When $X_{0,d}<1/4$, i.e., $Z_0$ is close to the boundary, recall our notation $\omega(Z_0,1)=Q(Z_0,1)\cap (0,T)\times\R^d_+$. Denote $\hat{Z}_0 = (X_0',0,T_0)$ to be the projection of $Z_0$ on the boundary and $\hat{z}_0^k$ to be projection of $z_0^k$. Note when $k$ is large,  $\lambda_k r\ll X_{0,d}$, hence
     $$
     Q(z_0^k,\lambda_k r)\subset Q(Z_0,X_{0,d})\subset Q^+(\hat{Z}_0,2X_{0,d}) \subset Q^+(\hat{Z}_0,1/2)\subset \w(Z_0,1).
     $$
     From the proofs of Corollaries \ref{col_CD} and \ref{col6_half}, we know that
     \begin{equation}
     \label{eqn_60}
     D(r,z_0,u_k,p_k) = D(\lambda_k r, z_0^k,u,p) \leq N D(X_{0,d},Z_0,u,p)+C,
     \end{equation}
     \begin{equation}
     \label{eqn_61}
     D^+(2X_{0,d},\hat{Z}_0,u,p)\leq N D^+(1/2,\hat{Z}_0,u,p)+C\leq ND^+(1,Z_0,u,p)+C\leq C.
     \end{equation}
     We use (\ref{presure_bd}) in the last inequality.
     Moreover, we have
      \begin{equation}
      \label{eqn_62}
     	\int_{ Q(Z_0,X_{0,d})}\abs{p-(p)_{ Z_0,X_{0,d}}}^{2-1/d}\, dz\leq N\int_{Q^+(\hat{Z}_0,2X_{0,d})}\abs{p-(p)_{\hat{Z}_0,2X_{0,d}}}^{2-1/d}\,dz,
     \end{equation}
     which implies
     $$ D(X_{0,d},Z_0,u,p) \leq N D^+(2X_{0,d},\hat{Z}_0,u,p).$$
     Together with  (\ref{eqn_60}) and (\ref{eqn_61}), we again deduce that $D(r,z_0,u_k,p_k)$ is uniformly bounded.

     \textsl{iii)} $\W=\R^d_+$ and  $X_{0,d}=0$, that is the half-space case when $Z_0$ lies on the boundary.  The domain of $u_k$ will expand to the half space, therefore $z_0\in(-\infty,0]\times \bR^d_+$. We compare the radius of the cylinder against the distance from $x_0$ to the boundary.  When $ r\geq x_{0,d}$, we have $\w(z_0^k,\lambda_k r)\subset Q^+(\hat{z}_0^k,2\lambda_k r) \subset Q^+(Z_0,1)$ when $k$ is large. By (\ref{presure_bd}), (\ref{eqn_62}),  and the proof of Corollary \ref{col6_half}, we have
     \begin{align*}
     	D^+(r,z_0,u_k,p_k) &= D^+(\lambda_k r,z_0^k,u,p) \nonumber \\
     	& \leq ND^+(2\lambda_k r,\hat{z}_0^k,u,p) \nonumber \\
     	& \leq ND^+(1,Z_0,u,p)+C \nonumber \\
     	& \leq C.
     \end{align*}
      When $ r\leq x_{0,d}$, we have
      $$
      \w(z_0^k,\lambda_k r)=Q(z_0^k,\lambda_k r)\subset Q(z_0^k,\lambda_k x_{0,d})\subset Q^+(\hat{z}_0^k,2\lambda_k x_{0,d}) \subset Q^+(Z_0,1)
      $$
      when $k$ is large.  From (\ref{presure_bd}), (\ref{eqn_62}), and the proofs of Corollaries \ref{col_CD}, and \ref{col6_half}, we know that
     \begin{align*}
     D(r,z_0,u_k,p_k)& = D(\lambda_k r,z_0^k,u,p) \nonumber \\
     &\leq ND(\lambda_k x_{0,d},z_0^k,u,p)+C \nonumber \\
     &\leq ND^+(2\lambda_k x_{0,d},\hat{z}_0^k,u,p)+C \nonumber \\
     &\leq ND^+(1,Z_0,u,p)+C \nonumber \\
     &\leq C.
     \end{align*}
Therefore, we have proved that $D^+(r,z_0,u_k,p_k)$ is uniformly bounded by  the $L^t_{\infty}L_d^x$ condition in (\ref{eqn_Ldcondition}) and  the local pressure condition in (\ref{presure_bd}) .
\end{proof}

Next we want to show, up to passing to a subsequence, $\{(u_k,p_k)\}_{k=1}^{\infty}$ converge to a limiting solution $(u_{\infty},p_{\infty})$. We modify \cite[Proposition 3.5]{Dong1} and state the results on $\R^d$ in next proposition. These results can be easily extended to $\R^d_+$.  To make the statement concise, we hereby introduce the following notation: $L_{p,\text{unif}} (\W_T)$, which means that the $L_{p}$ norm in $Q(z_0,1)\cap \W_T$ for any $z_0\in \W_T$ are uniformly bounded independent of the choice of $z_0$.
\begin{prop}
	\label{prop_ukconverge}
	i) There is a subsequence of {$(u_k,p_k)$}, which is still denoted by {$(u_k,p_k)$}, such that
\begin{equation}
\label{eqn51}
u_k\rightarrow u_{\infty} \text{ in } C([t_0-1/4^2,t_0];L_{q_1}(B(x_0,1/4))),
\end{equation}
\begin{equation}
\label{eqn52}
p_k\rightharpoonup p_{\infty} \text{ weakly in } L_{2-\frac{1}{d}}(Q(z_0,1/4)).
\end{equation}
for any $z_0\in (-\infty,0]\times\R^d
$ and $q_1\in [1,d)$.

ii) Furthermore, $(u_{\infty},p_{\infty})$ is a suitable weak solution of (\ref{NS1}) in $(-\infty,0)\times\R^d
$, and
$$ u_{\infty}\in L^t_{q_2}L_d^x((-T_1,0)\times\R^d
), \quad p_{\infty}\in L_{2-\frac{1}{d},\text{unif}} ((-T_1,0)\times\R^d
).$$
for any $T_1>0$ and $q_2\in[1,\infty)$.
	
\end{prop}
\begin{proof}
	First we fix a $z_0\in(-\infty,0]\times\R^d$. By the previous Lemma \ref{lem_bdd}, $p_k$'s have a uniform bound of the $L_{2-\frac{1}{d}}(Q(z_0,1))$ norm, so there is a subsequence, which is still denoted by $\{p_k\}$, such that (\ref{eqn52}) holds. Similarly,
	\begin{equation*}
	\norm{u_k}_{L_{\infty}^tL_d^x(Q(z_0,1))}\leq \norm{u_k}_{L_{\infty}^tL_d^x((t_0-1,t_0)\times\R^d)}\leq N,
	\end{equation*}
	where $N$ is independent of $k$. By Lemmas \ref{lemAECD} and \ref{lem_bdd}, we have
	\begin{equation*}
	A(1/2,z_0,u_k,p_k)+E(1/2,z_0,u_k,p_k)\leq N.
	\end{equation*}
	From (\ref{eqn_Ldcondition}) and the weak continuity of Leray-Hopf weak solutions, we can conclude that
	\begin{equation}
	\label{eqn59}
	\norm{u_k(t,\cdot)}_{L_d(B(z_0,1/2))}\leq N,
	\end{equation}
	for each  $t\in [-\lambda_k^{-2}T_0,0)$. By using Lemma \ref{lem2} with $q=2d/(d-2)$ and $r=1/2$, we have
	\begin{equation*}
		\norm{u_k}_{L_2^tL_{2d/(d-2)}^x(Q(z_0,1/2))}\leq N,
	\end{equation*}
which together with (\ref{eqn59}) and H\"{o}lder's inequality yields
\begin{equation*}
	\norm{u_k}_{L_4(Q(z_0,1/2))}\leq N, \quad \norm{u_k\cdot \nabla u_k}_{L_{4/3}(Q(z_0,1/2))}\leq N.
 \end{equation*}
 Following  the coercive estimate for the Stokes system (see, for instance, \cite{Refer29}) we have
\begin{equation*}
\partial_t u_k,\ D^2 u_k,\ \nabla p_k\in L_{4/3}(Q(z_0,1/4))
\end{equation*}
with uniform norms.
Therefore, we can find a subsequence still denoted by $\{u_k\}$ such that
\begin{equation*}
	u_k\rightarrow u_{\infty} \text{ in } C([t_0-1/4^2,t_0];L_{4/3}(B(x_0,1/4))).
\end{equation*}
This together with (\ref{eqn59}) gives (\ref{eqn51}) by using H\"{o}lder's inequality. To finish the proof of Part i), it suffices to use a Cauchy diagonal argument. Part ii) then follows from Part i)   and the fact that  $p_k$'s have a uniform bound of the $L_{2-\frac{1}{d}}$ norm in $Q(z_0,1)$, which is independent of the choice of $z_0$.
\end{proof}

\begin{corollary}\label{col_Csmall}
	When $d\geq 4$, for  any $\epsilon>0$, $r>0$, and $T_1\geq 1$, we can find $R\geq 1$ such that, for any $z_0\in (-T_1-1,0]\times(\mathbb{R}^d_{(+)}\setminus B_{R+1}^{(+)})$,
	\begin{equation}
	\label{eqn_limsupC}
	\limsup_{k\rightarrow \infty}C^{(+)}(r,z_0,u_k,p_k)\leq \epsilon,
	\end{equation}
\end{corollary}
\begin{proof}
For simplicity let us assume $\Omega=\mathbb{R}^d$, as the other case is similar.
Due to Proposition \ref{prop_ukconverge} ii), for any $r>0$, and $T_1\geq 1$, we can find $R$ large such that
\begin{equation*}
	\frac{1}{r^{2}}\int_{(-T_1-2,0)\times (\mathbb{R}^d\setminus B_R)}\abs{u_{\infty}}^d\, dz
\end{equation*}
is sufficiently small. Thus by H\"{o}lder's inequality when $d\geq 4$, for any $z_0\in (-T_1-1,0]\times(\mathbb{R}^d\setminus B_{R+1})$,
\begin{equation*}
		\frac{1}{r^{d+2/d-2}}\int_{Q(z_0,r)}\abs{u_{\infty}}^{2(2-1/d)}\, dz
\end{equation*}
is sufficiently small. This together with Proposition \ref{prop_ukconverge} i) proves (\ref{eqn_limsupC}).
\end{proof}

\begin{lemma}
	\label{lem_AECDsmall}
	When $d\geq 4$, for any $\epsilon_1>0$ and $T_1\leq 1$, we can find $R\geq 1$ and $r_0>0$ such that, for any  $z_0\in (-T_1-1,0]\times(\mathbb{R}^d_{(+)}\setminus B^{(+)}_{R+2})$,
	\begin{align}	\label{eqn_CDsmall}
		\begin{split}
				\limsup_{k\rightarrow \infty}\left(A^{(+)}(r_0,z_0,u_k,p_k)+E^{(+)}(r_0,z_0,u_k,p_k)\right.\\
				\left.+C^{(+)}(r_0,z_0,u_k,p_k)+D^{(+)}(r_0,z_0,u_k,p_k)\right)\leq \epsilon_1.
		\end{split}
	\end{align}
\end{lemma}
\begin{proof}
	This lemma is to improve the boundedness property of $C$ and $D$ we achieved in Lemma \ref{lem_bdd} to smallness.
We first prove the interior case.
From  Lemma \ref{lem_bdd}, we have
		\begin{equation*}
	\limsup_{k\rightarrow \infty}D(r,z_0,u_k,p_k)\leq N,
	\end{equation*}
	for any $z_0\in (-\infty,0]\times \mathbb{R}^d$.
	 Now let $r=1$. From  Corollary \ref{col_Csmall}, for any $\epsilon_0>0$, we can find $R>0$ such that  for any $z_0\in (-T_1-1,0]\times(\mathbb{R}^d\setminus B_{R+2})$,
	\begin{equation*}
	\limsup_{k\rightarrow \infty}C(1,z_0,u_k,p_k)\leq \epsilon_0.
	\end{equation*}
	By Lemma \ref{lem_CD}, we immediately know that there exists $\gamma>0$  such that
	\begin{equation}\label{eqn_2}
	\limsup_{k\rightarrow \infty}D(\gamma,z_0,u_k,p_k)\leq \frac{N\epsilon_0}{\gamma^{d+2/d-2}}+N\gamma^{4-3/d}.
	\end{equation}
	Then
	\begin{equation}
	\label{eqn_11}
	\limsup_{k\rightarrow \infty}C(\gamma,z_0,u_k,p_k)\leq \frac{\epsilon_0}{\gamma^{d+2/d-2}}.
	\end{equation}
	 Using Lemma \ref{lemAECD} we can get
	\begin{equation}\label{eqn_3}
	\limsup_{k\rightarrow \infty}(A(\gamma/2,z_0,u_k,p_k)+E(\gamma/2,z_0,u_k,p_k))\leq N\left(\frac{\epsilon_0}{\gamma^{d+2/d-2}}\right)^{1/4}.
	\end{equation}
	
We  add (\ref{eqn_11}), (\ref{eqn_2}), and (\ref{eqn_3}) together to obtain
	\begin{align}
	\label{eqn_4}
	&\limsup_{k\rightarrow \infty}(A(\gamma/2,z_0,u_k,p_k)+E(\gamma/2,z_0,u_k,p_k)+C(\gamma/2,z_0,u_k,p_k)+D(\gamma/2,z_0,u_k,p_k)) \nonumber \\
	 &\leq N\left(\frac{\epsilon_0}{\gamma^{d+2/d-2}}\right)^{1/4}
+\frac{N\epsilon_0}{\gamma^{d+2/d-2}}+N\gamma^{4-3/d} .
	\end{align}
	For any $\epsilon_1>0$, fix some $\gamma$ such that $N\gamma^{4-3/d}<\epsilon_1/2$. Next by choosing $R>0$ large enough, we can make $\epsilon_0$ small such that
$$
N\left(\frac{\epsilon_0}{\gamma^{d+2/d-2}}\right)^{1/4}
+\frac{N\epsilon_0}{\gamma^{d+2/d-2}}<\epsilon_1/2.
$$
Letting $r_0=\gamma/2$, by (\ref{eqn_4}) we have proved (\ref{eqn_CDsmall}) for the interior case.
	
	For the boundary case, we do a similar discussion on different shapes of $\w (z_0^k,\lambda_k r)$ as we did in the proof of  Lemma \ref{lem_bdd}.	We hereby omit the repeated proof.
\end{proof}
Next we show that $u_{\infty}$ is identically equal to zero. We modify the proof of  \cite[Proposition 5.3]{Dong1} by replacing  Schoen's trick with the H\"{o}lder continuity proved  in Sections \ref{sec_holder_1} and \ref{sec_holder_2}. We state and prove the following proposition for $\Omega=\mathbb{R}^d_+$. The results readily generalize to problem on $\Omega=\mathbb{R}^d$.
\begin{prop}
	\label{prop_u_infty}
	Under the assumptions of Theorem \ref{thm_main}, let $(u_{\infty},p_{\infty})$ be the suitable weak solution constructed in this section. Then
	\begin{equation*}
	u_{\infty}(t,\cdot)\equiv 0 \quad \forall t\in (-\infty,0).
	\end{equation*}
\end{prop}
\begin{proof}
	Let $\hat{\epsilon}_0$ be the  constant in Theorem \ref{thm_holder_half}. Note that we can assume $\hat{\epsilon}_0$ is smaller than $\epsilon_0$ in Theorem \ref{thm_holder_whole}.  Fix some $T_1\geq 1$. Owing to Lemma \ref{lem_AECDsmall}, we can find $R\geq 1$ and $r_0>0$ such that for any $z_0\in [-T_1-1,0]\times(\mathbb{R}^d_+\setminus B_{R+1}^+)$ ,
	$$	\limsup_{k\rightarrow \infty}(A^+(r_0,z_0,u_k,p_k)+E^+(r_0,z_0,u_k,p_k)+D^+(r_0,z_0,u_k,p_k))\leq \hat{\epsilon}_0.$$ Thus Theorems \ref{thm_holder_whole} and \ref{thm_holder_half} yields that
	$$\limsup_{k\rightarrow \infty}\abs{u_k(z_0)}\leq N(d)$$
	for a.e. $z_0\in[-T_1-1,0)\times (\mathbb{R}^d_+\setminus B^+_{R+2})$. By Proposition \ref{prop_ukconverge}, we obtain
	$$\abs{u_{\infty}(z_0)}\leq N(d)$$
	for a.e. $z_0\in[-T_1-1,0)\times (\mathbb{R}^d_+\setminus B^+_{R+2})$. Upon using the regularity results for linear Stokes systems, one can estimate higher derivatives
	\begin{equation}
		\label{eqn_stoke}
		\abs{D^ju_{\infty}(z_0)}\leq N(d,j)
	\end{equation}
	for any $j\geq 1$ and a.e. $z_0\in[-T_1-1,0)\times (\mathbb{R}^d_+\setminus B^+_{R+3})$.
	
	We now claim $u_{\infty}(0,\cdot)=0$ by adapting the argument in the proof of  \cite[Theorem 1.4]{Seregin1}. For any $x_0\in \mathbb{R}^d_+$, by using (\ref{eqn51}),
	\begin{align*}
	&\int_{\Omega(x_0,1/4)}\abs{u_{\infty}(0,x)}\, dx\\
	&\leq \int_{\Omega(x_0,1/4)}\abs{u_k(0,x)-u_{\infty}(0,x)}\, dx+\int_{\Omega(x_0,1/4)}\abs{u_k(0,x)}\, dx\\
	&\leq \norm{u_k-u}_{C([-1/4^2,0];L_1(\Omega(x_0,1/4)))}+N(d)\left(\int_{\Omega(x_0,1/4)}\abs{u_k(0,x)}^d\,dx\right)^{1/d}\\
	&=\norm{u_k-u}_{C([-1/4^2,0];L_1(\Omega(x_0,1/4)))}+N(d)\left(\int_{\Omega(\lambda_k x_0/4+X_0,\lambda_k/4)}\abs{u(T_0,y)}^d\,dy\right)^{1/d}.	
	\end{align*}
	The right-hand side of the above inequality goes to zero as $k\rightarrow \infty$, which proves the claim.
	
	Because of (\ref{eqn_stoke}), the vorticity $\omega=$ curl $u_{\infty}$ satisfies the differential inequality
	$$\abs{\partial_t\omega-\Delta\omega}\leq N(\abs{\omega}+\abs{\nabla \omega}),$$
	on $(-T_1,0)\times(\mathbb{R}^d_+\setminus B^+_{R+3})$.
	 We apply the half-space backward uniqueness theorem proved in \cite{Seregin1, Seregin2} on the open half space  $\{x_d>R+3\}$ to reach
	\begin{equation}
	\label{eqn_vorticity}
	\omega(z)=0 \quad \text {on } (-T_1,0]\times\{x_d>R+3\}.
	\end{equation}
	
	Now we fix a $t_0\in(-T_1,0)$. Take an increasing  sequence $\{t_k\}_{k=0}^{\infty}\subset (-T_1,0)$ converging to $t_0$. For each $k$, we consider equation (\ref{NS1}) with initial data $u_{\infty}(t_k,\cdot)$. By Proposition \ref{prop_StrongSolva}, one can locally find a strong solution
	$$v_k\in C([t_k,t_k+\delta_k);L_d(\mathbb{R}^d_+))$$
	for some small $\delta_k$, and $v_k(t,\cdot)$ is spatial analytic for $t\in (t_k,t_k+\delta_k)$. We may assume that $t_k+\delta_k<t_0$. By the weak-strong uniqueness, $v_k\equiv u_{\infty}$ for $t\in [t_k,t_k+\delta_k)$. Therefore, $\omega(t,\cdot)$ is also spatial analytic for $t\in(t_k,t_k+\delta_k)$. Because of (\ref{eqn_vorticity}), we get
	$$\omega(z)=0 \quad \text {on } (t_k,t_k+\delta_k)\times\mathbb{R}^d_+,$$
	which implies that $u_{\infty}\equiv 0$ in the same region. In particular, we can take a sequence $\{s_k\}$ such that $t_k<s_k<t_k+\delta_k$. Then $\{s_k\}$  converges to $t_0$ and
	$$u_{\infty}(s_k,\cdot)\equiv 0.$$
	This together with the weak continuity of $u_{\infty}$ yields that $u_{\infty}(t_0,\cdot)\equiv 0$. Since $t_0\in (-T_1,0)$  and $T_1\geq 1$ are both arbitrary, we complete the proof of the theorem.
\end{proof}

In the following, we prove the main theorem in the setting that $\Omega=\mathbb{R}^d_+$. The proof also works for the problem on $\Omega=\mathbb{R}^d$.
Again this proof is a modification of \cite[Section 5]{Dong1} by replacing  Schoen's trick with the H\"{o}lder continuity proved in Sections \ref{sec_holder_1} and \ref{sec_holder_2}.
\begin{proof}[Proof of Theorem \ref{thm_main}]
	We prove the theorem in four steps.
	
	\textsl{Step 1.}
First we show that $u$ is regular for $t\in (0,T]$. Owning to Propositions \ref{prop_ukconverge} and \ref{prop_u_infty},
$$u_k\rightarrow 0 \text { in }C([-1/4^2,0];L_{2(2-1/d)}(B^+(1/4))).$$
Hence
$$\limsup_{k\rightarrow \infty} C^+(r,z_0,u_k,p_k)=0$$
for $r\in(0,1/4)$.
Also recall that $D^+(r,z_0,u_k,p_k)$ has a uniform bound for $r\in(0,1)$. Following the proof of Lemma \ref{lem_AECDsmall} we have: for any $\epsilon>0$, there is a $r_0>0$ small and a positive integer $k_0$ such that, for any $z_0\in(-2,0]\times B^+(2)$,
$$A^+(r_0,z_0,u_{k_0},p_{k_0})+E^+(r_0,z_0,u_{k_0},p_{k_0})+D^+(r_0,z_0,u_{k_0},p_{k_0})\leq \epsilon.$$
We choose $\epsilon$ sufficiently small and apply Theorem \ref{thm_holder_half} to get, for some  $\tilde{r}>0$,
$$\sup_{(-\tilde{r}^2,0)\times B^+(\tilde{r})} \abs{u_{k_0}}<\infty,$$
which implies that
$$\sup_{\omega(Z_0,\lambda_{k_0}\tilde{r})}\abs{u}<\infty.$$
This contradicts the assumption that $Z_0=(T_0,X_0)$ is a blowup point. Therefore, $u$ is regular for $t\in (0,T]$.

\textsl{Step 2.} We bound the sup norm of $u$ in this step. Fix some small $\delta\in(0,T)$. Since
$$\norm{u}_{L_{\infty}^tL_d^x((0,T)\times\mathbb{R}^d_+)}\leq N,\quad \norm{p}_{L_{2-1/d,\text{unif}}((0,T)\times\mathbb{R}^d_+)}\leq N,$$
by the same reasoning as at the beginning of the proof of Proposition \ref{prop_u_infty}, we see that there exists a large $R\geq 1$ such that
\begin{equation}
\label{eqn_uouter}
\sup_{[\delta,T)\times (\mathbb{R}^d_+\setminus B^+(R))}\abs{u}\leq N.
\end{equation}
Next we estimate the sup norm of $u$ in $[\delta,T)\times B^+(R)$. Fix a $z_0=(t_0,x_0)$ in  $[\delta,T]\times\bar{B}^+(R)$. In the construction of $u_k$, we replace $(T_0,X_0)$ by $(t_0,x_0)$. By the same reasoning as in the first step, for some $\tilde{r}=\tilde{r}(t_0,x_0)>0$, we have
$$\sup_{\omega(z_0,\tilde{r})}\abs{u}\leq N(t_0,x_0).$$
By the compactness of  $[\delta,T]\times\bar{B}^+(R)$, it holds that
$$\sup_{ [\delta,T]\times\bar{B}^+(R)}\abs{u}\leq N.$$
This together with (\ref{eqn_uouter}) yields
$$\sup_{ [\delta,T]\times \mathbb{R}^d_+}\abs{u}\leq N.$$

\textsl{Step 3.} Since we can choose the $\delta$ in step 2 to be arbitrarily small, we can then utilize the local strong solvability of (\ref{NS1}) to find some $T_1>\delta$ such that $u\in L_{d+2}((0,T_1)\times\mathbb{R}^d_+)$. From Step 2, for $t\in[T_1,T]$ the solution is uniformly bounded and belongs to $L_{\infty}^tL_d^x((T_1,T)\times\mathbb{R}^d_+)$, thus $u\in L_{d+2}((0,T)\times\mathbb{R}^d_+)$. The uniqueness follows from the Ladyzhenskaya-Prodi-Serrin criterion.

\textsl{Step 4.} Now it remains to prove (\ref{eqn_main_thm2}). We use a scaling argument. Let $\lambda>0$ be a constant to be specified later. We define
	$$u_{\lambda}(t,x) = \lambda u(\lambda^2t,\lambda x),\quad p_{\lambda}(t,x) = \lambda^2 p(\lambda^2 t,\lambda x).$$
	Then $(u_{\lambda},p_{\lambda})$ is also a Leray-Hopf weak solution of (\ref{NS1}) in $(0,\infty)\times \mathbb{R}^d_+$, and $u_{\lambda}$ satisfies (\ref{eqn_Ldcondition}) with the same constant $K$ due to the scale invariant property.
Due to the weak continuity of Leray-Hopf weak solutions,
\begin{equation}
\label{eqn53}
\norm{u_{\lambda}(t,\cdot)}_{L_d(\mathbb{R}^d_+)}\leq N,
\end{equation}
for each $t\in(0,\infty)$. By using Lemma \ref{lem2} with $q=2d/(d-2)$ and $r=\infty$, we have
\begin{equation}
\label{eqn54}
\norm{u_{\lambda}}_{L^t_2L^x_{2d/(d-2)}((0,\infty)\times\mathbb{R}^d_+)}\leq N.
\end{equation}
Putting together (\ref{eqn53}) and (\ref{eqn54}) and using H\"{o}lder's inequality yield
\begin{equation*}
\norm{u_{\lambda}}_{L_4((0,\infty)\times\mathbb{R}^d_+)}\leq N.
\end{equation*}
Thus  there  exists $T>0$ such that $$\int_{(T,\infty)\times\bR_+^d}\abs{u_{\lambda}}^4\, dz$$
	is sufficiently small, which together with H\"{o}lder's inequality implies the smallness of $C^+(1,z_0,u_{\lambda},p_{\lambda})$ for $z_0\in (T,\infty)\times\mathbb{R}^d_+$.
Again let $\hat{\epsilon}_0$ denote the  constant in Theorem  \ref{thm_holder_half}. Following the argument in Lemma \ref{lem_AECDsmall} we can find a large $T=T_{\lambda}$ and $r_0>0$ such that
$$ A^+(r_0,z_0,u_{\lambda},p_{\lambda})+E^+(r_0,z_0,u_{\lambda},p_{\lambda})+D^+(r_0,z_0,u_{\lambda},p_{\lambda})\leq \hat{\epsilon}_0$$
for any $z_0\in[T,\infty)\times\mathbb{R}^d_+$. Owing to Theorems \ref{thm_holder_whole} and \ref{thm_holder_half}, we conclude
$$\sup_{[T,\infty)\times\mathbb{R}^d_+}\abs{u_{\lambda}(z)}\leq N,$$
where $N=N(d)$ is independent of $\lambda$. Therefore,
$$\sup_{[\lambda^2T,\infty)\times\mathbb{R}^d_+}\abs{u(z)}\leq N/\lambda.$$
Sending $\lambda\rightarrow\infty$ yields the desired result.
\end{proof}

\section{Proof of Theorem \ref{thm_main3}}
\label{sec_thm2}
In this section, we prove the local result mentioned in Theorem \ref{thm_main3}. Most part of the proof remains the same with last section. We omit some repeated details in similar proofs from last section. We again start with the blow-up procedure: Suppose $(u,p)$ is a pair of Leray-Hopf weak solution  (\ref{NS1})-(\ref{NS2}) on $\R^d_+$ and we study the local problem in $Q^+$. Correspondingly, we modify the notation $\w(z,r) = Q(z,r) \cap Q^+$. Suppose $T_0$ is the first blowup time of $u$ in $\overline{Q^+_{1/4}}$, and $Z_0=(T_0,X_0)=(T_0,X_{0,1},X_{0,2},\ldots,X_{0,d})=(T_0,X'_0,X_{0,d})$ is a singular point in $\overline{Q^+_{1/4}}$. Take a decreasing sequence $\{\lambda_k\}$ converging to 0 and rescale the pair $(u,p)$ at time $T_0$. By defining
$$u_k(t,x) = \lambda_k u(T_0+\lambda_k^2t,X_0+\lambda_kx),\quad p_k(t,x) = \lambda_k^2p(T_0+\lambda_k^2t,X_0+\lambda_kx),$$
for each $k=1,2,\ldots$, $(u_k,p_k)$ is a suitable weak solution of (\ref{NS1})-(\ref{NS2}) 
for $t\in (-\lambda_k^{-2}T_0,0)$.
The first observation is the property of uniform boundedness of $C^+$ and $D^+$ after the rescaling.
\begin{lemma}
	\label{lem_bdd_local}
	Under the conditions in Theorem \ref{thm_main3}, there exists $N>0$  such that for any $z_0\in (-\infty,0]\times \R^d_{(+)}$ and $0<r\leq 1$,
	\begin{equation*}
	\limsup_{k\rightarrow \infty} C^{+}(r,z_0,u_k,p_k)\leq N,
	\end{equation*}
	and
	\begin{equation}
	\label{eqn_Dbdd_local}
	\limsup_{k\rightarrow \infty}D^{+}(r,z_0,u_k,p_k)\leq N.
	\end{equation}
\end{lemma}
\begin{proof}
	Following the proof of Lemma \ref{lem_bdd}, we easily have
	\begin{equation*}
	C^+(\lambda_kr,z_0^k,u,p)\leq N\norm{u}_{L_{\infty}^tL_d^x(\w(z_0^k,\lambda_kr))}\leq N,
	\end{equation*}
To prove (\ref{eqn_Dbdd_local}), we consider three cases on whether $Z_0$ is on the flat boundary and the shape of the $\w(z_0^k,\lambda_k r)$:

\textsl{i)} Consider when $Z_0\in Q^+_{1/4}$, i.e.,  $X_{0,d}>0$. Based on the blow-up procedure, the $u_k$'s are defined on a larger and larger domain which eventually expands to the whole space, so we can always assume $u_k$ has definition on $Q(z_0,r)$ when $k$ is large. Denote $\hat{Z}_0 = (X_0',0,T_0)$ to be the projection of $Z_0$ on the flat boundary and $\hat{z}_0^k$ to be the projection of $z_0^k$. Note when $k$ is large,  $\lambda_k r\ll X_{0,d}$, we have $Q(z_0^k,\lambda_k r)\subset Q(Z_0,X_{0,d})\subset Q^+(\hat{Z}_0,2X_{0,d}) \subset Q^+(\hat{Z}_0,1/2)\subset Q^+$. From Corollaries \ref{col_CD} and \ref{col6_half}, we know that
\begin{align*}
D(r,z_0,u_k,p_k) &= D(\lambda_k r, z_0^k,u,p) \nonumber \\
& \leq N D(X_{0,d},Z_0,u,p)+C \nonumber\\
& \leq N D^+(2X_{0,d},\hat{Z}_0,u,p)+C \nonumber \\
& \leq N D^+(1/2,\hat{Z}_0,u,p)+C\nonumber \\
& \leq N\norm{p}_{L_{2-1/d}(Q^+)}^{2-1/d}+C.
\end{align*}

\textsl{ii)} Consider when $Z_0\in Q^+_{1/4}$,  $X_{0,d}=0$, and $r\geq x_{0,d}$. When $k$ is large, we have $ \w(z_0^k,\lambda_k r)\subset Q^+(\hat{z}^k_0,2\lambda_k r) \subset Q^+(Z_0,1/2)\subset Q^+$. From the proof of Corollary \ref{col6_half} we have
\begin{align*}
	D^+(r,z_0,u_k,p_k) &= D^+(\lambda_k r, z_0^k, u,p) \nonumber \\
	& \leq ND^+(2\lambda_k r,\hat{z}_0^k,u,p)\nonumber \\
	& \leq ND^+(1/2,Z_0,u,p)+C\nonumber \\
	& \leq N\norm{p}_{L_{2-1/d}(Q^+)}^{2-1/d}+C.
\end{align*}

\textsl{iii)} Consider when $Z_0\in Q^+_{1/4}$,  $X_{0,d}=0$ and $r< x_{0,d}$. When $k$ is large, we have $\w(z_0^k,\lambda_k r)=Q(z_0^k,\lambda_k r)\subset Q(\hat{z}_0^k,\lambda_k x_{0,d}) \subset Q^+(Z_0,1/2)\subset Q^+$. From the proofs of Corollaries \ref{col_CD} and \ref{col6_half} we have
\begin{align*}
	D^+(r,z_0,u_k,p_k) &= D(\lambda_k r, z_0^k, u,p) \nonumber \\
	& \leq ND(\lambda_k x_{0,d},z_0^k,u,p)+C\nonumber \\
	& \leq ND(2\lambda_k x_{0,d},\hat{z}_0^k,u,p)+C\nonumber \\
	& \leq ND^+(1/2,Z_0,u,p)+C\nonumber \\
	& \leq N\norm{p}_{L_{2-1/d}(Q^+)}^{2-1/d}+C.
\end{align*}
Therefore, $D^+(r,z_0,u_k,p_k)$ is uniformly bounded by $\norm{p}_{L_{2-1/d}(Q^+)}$ plus some constant.
\end{proof}

	After we proved the boundedness, we can again prove the results of  Proposition \ref{prop_ukconverge}, Corollary \ref{col_Csmall}, Lemma \ref{lem_AECDsmall} and Proposition \ref{prop_u_infty} in the same fashion as previous. Here we show the proof of Theorem \ref{thm_main3} in the boundary case $X_{0,d}=0$. The interior case is similar.
\begin{proof}[Proof of Theorem \ref{thm_main3}]
	Owing to Propositions \ref{prop_ukconverge} and \ref{prop_u_infty}, we have
	$$u_k\rightarrow 0 \text { in }C([-1/4^2,0];L_{2(2-1/d)}(B^+(1/4))).$$
	Hence
	$$\limsup_{k\rightarrow \infty} C^+(r,z_0,u_k,p_k)=0$$
	for $r\in(0,1/4)$.
	Also recall that $D^+(r,z_0,u_k,p_k)$ has a uniform bound for $r\in(0,1/4)$. Following the proof of Lemma \ref{lem_AECDsmall} we have: for any $\epsilon>0$, there is a $r_0>0$ small and a positive integer $k_0$ such that, for any $z_0\in(-1/4^2,0]\times B^+(1/4)$,
	$$A^+(r_0,z_0,u_{k_0},p_{k_0})+E^+(r_0,z_0,u_{k_0},p_{k_0})+D^+(r_0,z_0,u_{k_0},p_{k_0})\leq \epsilon,$$
	which implies
	$$A^+(\lambda_{k_0} r_0,z_0^{k_0},u,p)+E^+(\lambda_{k_0} r_0,z_0^{k_0},u,p)+D^+(\lambda_{k_0} r_0,z_0^{k_0},u,p)\leq \epsilon,$$
	We choose $\epsilon$ sufficiently small and apply Theorem \ref{thm_holder_half} to deduce that $u$ is H\"{o}lder continuous around $Z_0$  since $z_0^{k_0}$ is an arbitrary point around $Z_0$. We reach a contradiction and  finish the proof.
\end{proof}

\bibliographystyle{amsplain}
\providecommand{\bysame}{\leavevmode\hbox to3em{\hrulefill}\thinspace}
\providecommand{\MR}{\relax\ifhmode\unskip\space\fi MR }
\providecommand{\MRhref}[2]{%
  \href{http://www.ams.org/mathscinet-getitem?mr=#1}{#2}
}
\providecommand{\href}[2]{#2}

\end{document}